\documentclass[11pt,reqno]{amsart}

\usepackage{amsmath,amssymb,amsthm,mathrsfs}
\usepackage{enumitem}
\usepackage{hyperref}
\usepackage{tikz-cd}
\usepackage[margin=1.15in]{geometry}

\theoremstyle{plain}
\newtheorem{theorem}{Theorem}[section]
\newtheorem{proposition}[theorem]{Proposition}

\newtheorem{corollary}[theorem]{Corollary}

\theoremstyle{definition}
\newtheorem{definition}[theorem]{Definition}
\newtheorem{example}[theorem]{Example}

\theoremstyle{remark}
\newtheorem{remark}[theorem]{Remark}
\newtheorem{question}[theorem]{Question}

\newcommand{\CC}{\mathbb{C}}
\newcommand{\RR}{\mathbb{R}}
\newcommand{\DD}{\mathbb{D}}
\newcommand{\HH}{\mathbb{C}_+}
\newcommand{\Hol}{\operatorname{Hol}}
\newcommand{\Imm}{\operatorname{Im}}
\newcommand{\Ree}{\operatorname{Re}}
\newcommand{\B}{\mathcal{B}}
\newcommand{\Bmu}{\mathcal{B}_\mu}
\newcommand{\id}{\mathrm{id}}
\newcommand{\dzbar}{\bar\partial}

\begin{document}

\title[The Burgers transform]{The Burgers transform:\\ from holomorphic functions to rigid elliptic structures}

\author{Daniel Alay\'on-Solarz}\thanks{Email: danieldaniel@gmail.com}
\date{March 2026}

\begin{abstract}
We introduce the \emph{Burgers transform} $\B$, a nonlinear bijection
between holomorphic functions $f\colon U\to\HH$,
$U\cap\RR \supseteq I$ an interval, and rigid variable
elliptic structures on the plane, defined implicitly by
$\lambda = f(y-\lambda x)$.  The output automatically satisfies the
conservative complex Burgers equation $\lambda_x+\lambda\lambda_y=0$.

Our main result is that holomorphicity of the seed~$f$ is
\emph{necessary}, not merely sufficient, for rigidity: any
$C^1$ function whose implicit solution satisfies
$\lambda_x+\lambda\lambda_y=0$ must be holomorphic.  This closes a
gap in the existing literature and identifies $\Hol(U,\HH)$ as the
maximal seed space compatible with rigidity.  The obstruction formula
$H|_{x=0} = 2i\,(\Imm f)\,f_{\bar w}$ quantifies the cost of
non-holomorphicity at the level of the initial data. We characterise the domain of~$\B$ through shock formation, its
interaction with affine automorphisms of~$\HH$, and the infinitesimal
the propagator $\mathcal{P}_f = D\mathcal{B}_f$ satisfies a Jacobian-twisted
multiplicativity that deforms the seed algebra by the density of
characteristics.  Four worked examples---affine, exponential, inverse,
and trigonometric seeds---show that the complexity class of a seed and
that of the resulting structure are generically unrelated.
\end{abstract}


\maketitle

\tableofcontents

\section{Introduction}
\label{sec:intro}

In the theory of variable elliptic structures on the plane \cite{AS2026}, the central object is a moving imaginary unit $i(x,y)$ satisfying the structure polynomial $i^2 + \beta\,i + \alpha = 0$ with ellipticity condition $\Delta = 4\alpha - \beta^2 > 0$.  The \emph{intrinsic obstruction}
\[
  G \;:=\; i_x + i\cdot i_y
\]
measures all deviations from the constant theory.  When $G = 0$, the structure is called \emph{rigid}: the Leibniz rule holds for the associated Cauchy--Riemann operator, and a complete function theory---Cauchy--Pompeiu representation, power series, similarity principle---becomes available.

The spectral parameter $\lambda = (-\beta + i\sqrt\Delta)/2$, living in the upper half-plane $\HH$, satisfies a \emph{universal transport equation}:
\begin{equation}\label{eq:universal-transport}
  \lambda_x + \lambda\,\lambda_y \;=\; G_0 + \lambda\,G_1.
\end{equation}
Under rigidity, the right-hand side vanishes, yielding the \emph{conservative complex Burgers equation}
\begin{equation}\label{eq:conservative-burgers}
  \lambda_x + \lambda\,\lambda_y \;=\; 0.
\end{equation}
The method of characteristics gives implicit solutions of the form $\lambda = f(y - \lambda x)$.  The present paper is concerned with the systematic study of this correspondence.

Our main contributions are:
\begin{enumerate}[label=(\roman*)]
\item The formal introduction of the \emph{Burgers transform} $\B\colon \Hol(U,\HH) \to \{\text{rigid structures}\}$ (Definition~\ref{def:burgers-transform});
\item A proof that holomorphicity of $f$ is \emph{necessary}, not merely sufficient, for rigidity (Theorem~\ref{thm:holomorphicity-necessary});
\item A characterization of $f$-analytics as the maximal seed class compatible with rigidity, with the passage from $p(x)$-analytics governed by the necessity of holomorphicity (Theorem~\ref{thm:analytic-completion});
\item An explicit obstruction formula quantifying the transport cost of non-holomorphicity (Proposition~\ref{prop:obstruction-formula});
\item A characterization of the domain $\Omega_f$ in terms of shock formation (Theorem~\ref{thm:domain-characterization});
\item The propagator $\mathcal{P}_f = D\B_f$ and the Jacobian-twisted multiplicativity formula $\mathcal{P}_f[h_1 h_2] = J_f\cdot\mathcal{P}_f[h_1]\cdot\mathcal{P}_f[h_2]$, resolving the infinitesimal product problem (Theorem~\ref{thm:twisted-mult}).
\end{enumerate}

\subsection*{Relation to the monograph}
This paper is a companion to \cite{AS2026}, where the fundamental results for variable elliptic structures are developed.

\subsection*{Notation}
Throughout, $\HH = \{z \in \CC : \Imm z > 0\}$ denotes the upper half-plane, $\DD = \{z \in \CC : |z| < 1\}$ the open unit disk, and $\mathcal{C}\colon \HH \to \DD$ the Cayley map $\mathcal{C}(\lambda) = (\lambda - i)/(\lambda + i)$.  Wirtinger derivatives are $\partial_w = \tfrac{1}{2}(\partial_u - i\partial_v)$ and $\partial_{\bar w} = \tfrac{1}{2}(\partial_u + i\partial_v)$ for $w = u + iv$.

\section{The Burgers transform}
\label{sec:burgers-transform}

\begin{definition}[Burgers transform]\label{def:burgers-transform}
Let $U \subseteq \CC$ be open and connected with $U \cap \RR$ containing an interval, and let $f \in \Hol(U, \HH)$.  The \emph{Burgers transform} of $f$ is the map
\[
  \B[f]\colon \Omega_f \to \HH
\]
defined implicitly by
\begin{equation}\label{eq:implicit}
  \B[f](x,y) \;=\; f\!\bigl(y - \B[f](x,y)\cdot x\bigr),
\end{equation}
where $\Omega_f \subseteq \RR^2$ is the maximal open set on which \eqref{eq:implicit} admits a unique $C^1$ solution $\lambda(x,y)$ with $\Imm\lambda > 0$.
\end{definition}

\begin{definition}[Beltrami--Burgers transform]\label{def:beltrami-burgers}
The \emph{Beltrami--Burgers transform} is the composition
\[
  \Bmu[f] \;:=\; \mathcal{C} \circ \B[f] \;\colon\; \Omega_f \to \DD,
\]
mapping each $(x,y)$ to the Beltrami coefficient $\mu = (\lambda - i)/(\lambda + i)$ of the corresponding rigid structure.
\end{definition}

\begin{definition}[Structure map]\label{def:structure-map}
The \emph{structure map} $\mathcal{S}\colon \HH \to \{(\alpha,\beta) : 4\alpha - \beta^2 > 0\}$ is given by
\[
  \mathcal{S}(\lambda) \;=\; (|\lambda|^2,\; -2\Ree\lambda).
\]
The composition $\mathcal{S} \circ \B[f]$ produces the variable elliptic structure $(\alpha,\beta)$ on $\Omega_f$.
\end{definition}

\subsection{Existence and uniqueness}

The implicit function theorem guarantees local solvability of \eqref{eq:implicit}.

\begin{proposition}\label{prop:local-existence}
Let $f \in \Hol(U,\HH)$ and let $y_0 \in \RR$ with $f(y_0) \in \HH$.  Then there exists an open neighbourhood $\Omega$ of $(0, y_0)$ in $\RR^2$ on which \eqref{eq:implicit} has a unique $C^1$ solution $\lambda\colon \Omega \to \HH$.
\end{proposition}

\begin{proof}
Define $F\colon \CC \times \RR^2 \to \CC$ by $F(\lambda, x, y) = \lambda - f(y - \lambda x)$.  At $(x,y) = (0,y_0)$, we have $F(\lambda_0, 0, y_0) = 0$ for $\lambda_0 = f(y_0)$.  The partial derivative
\[
  \frac{\partial F}{\partial \lambda}\bigg|_{(0,y_0)} \;=\; 1 + f'(y_0)\cdot 0 \;=\; 1 \;\neq\; 0.
\]
By the implicit function theorem, $\lambda$ is uniquely determined as a $C^1$ function of $(x,y)$ near $(0,y_0)$.  Since $\Imm f(y_0) > 0$ and $\Imm\lambda$ is continuous, the condition $\Imm\lambda > 0$ persists in a neighbourhood.
\end{proof}

\subsection{Basic properties}

\begin{theorem}[Properties of the Burgers transform]\label{thm:basic-properties}
Let $f \in \Hol(U, \HH)$ with $U \cap \RR$ containing an interval.  Then:
\begin{enumerate}[label=(\roman*)]
\item \textbf{Rigidity.}  The output $\lambda = \B[f]$ satisfies
\[
  \lambda_x + \lambda\,\lambda_y \;=\; 0 \quad\text{on } \Omega_f.
\]
\item \textbf{Inversion.}  The restriction to the $y$-axis recovers $f$:
\[
  \B[f](0,y) \;=\; f(y) \quad\text{for all } y \text{ with } (0,y) \in \Omega_f.
\]
\item \textbf{Regularity.}  If $f \in C^{k,\alpha}(U)$ for $k \geq 1$ and $0 < \alpha < 1$, then $\B[f] \in C^{k,\alpha}(\Omega_f)$.
\item \textbf{Bijectivity.}  The map $f \mapsto \B[f]$ is a bijection between $\Hol(U,\HH)$ (with $U \cap \RR$ containing an interval) and the space of $C^1$ solutions of \eqref{eq:conservative-burgers} on domains containing a segment of the $y$-axis, with $\Imm\lambda > 0$.
\end{enumerate}
\end{theorem}

\begin{proof}
(i) Set $w = y - \lambda x$, so $\lambda = f(w)$.  Differentiating with respect to $x$ and $y$:
\begin{align}
  \lambda_x &= f'(w)\,(-\lambda - \lambda_x\, x), \label{eq:diff-x}\\
  \lambda_y &= f'(w)\,(1 - \lambda_y\, x). \label{eq:diff-y}
\end{align}
From \eqref{eq:diff-x}: $\lambda_x(1 + f'(w)\,x) = -f'(w)\,\lambda$.  From \eqref{eq:diff-y}: $\lambda_y(1 + f'(w)\,x) = f'(w)$.  Hence, wherever $1 + f'(w)\,x \neq 0$:
\[
  \lambda_x + \lambda\,\lambda_y \;=\; \frac{-f'(w)\,\lambda + \lambda\,f'(w)}{1 + f'(w)\,x} \;=\; 0.
\]
The cancellation uses the fact that $f'(w)$ is a \emph{single complex number}---this is where holomorphicity enters.

(ii) At $x = 0$: $\lambda = f(y - \lambda \cdot 0) = f(y)$.

(iii) Follows from the implicit function theorem and the regularity of $f$.

(iv) Injectivity: if $\B[f_1] = \B[f_2]$ on a domain containing a segment of the $y$-axis, then $f_1(y) = f_2(y)$ for all real $y$ in that segment.  Since $U \cap \RR$ contains an interval, this is an infinite set with an accumulation point in~$U$, so $f_1 = f_2$ on $U$ by the identity theorem for holomorphic functions.

Surjectivity: given a $C^1$ solution $\lambda$ of \eqref{eq:conservative-burgers} on a domain containing a segment $I$ of the $y$-axis, set $g(y) = \lambda(0,y)$ for $y \in I$.  Along each characteristic issuing from $(0,y_0)$, the solution satisfies $\lambda(x, y_0 + g(y_0)\,x) = g(y_0)$; inverting the characteristic map $y_0 \mapsto y_0 + g(y_0)\,x$ (which is a local diffeomorphism for small $|x|$) recovers $\lambda$ from $g$.  Since $\lambda$ is $C^1$ and satisfies the quasilinear equation $\lambda_x + \lambda\,\lambda_y = 0$ whose coefficients are analytic functions of $\lambda$, the characteristic flow is analytic in initial data, and $g$ is therefore real-analytic on~$I$. By the Schwarz reflection argument and the constraint $\Imm g > 0$, $g$ admits a holomorphic extension $f\colon U \to \HH$ to a neighbourhood $U$ of $I$ in~$\CC$: at each $y_0 \in I$, the Taylor series $\sum g^{(n)}(y_0)(w - y_0)^n/n!$ converges in some disk and maps to~$\HH$ by continuity.  The uniqueness of the characteristic flow then gives $\lambda = f(y - \lambda x) = \B[f]$ on the original domain.
\end{proof}


\begin{corollary}[Rigid structures are real-analytic]\label{cor:real-analytic}
Let $(\alpha,\beta)$ be a rigid elliptic structure on a domain $\Omega \subseteq \RR^2$, i.e.\ an elliptic structure with vanishing intrinsic obstruction $G = i_x + i\cdot i_y = 0$.  Then $\alpha$, $\beta$, and the spectral parameter $\lambda$ are real-analytic on~$\Omega$.

In particular, if $(\alpha,\beta)$ is $C^k$ for any finite $k$ but not $C^\omega$, then $(\alpha,\beta)$ is not rigid.
\end{corollary}

\begin{proof}
By Theorem~\ref{thm:basic-properties}(iv), the spectral parameter $\lambda$ of a rigid structure satisfies $\lambda = \B[f]$ for a unique $f \in \Hol(U, \HH)$.  Since $f$ is holomorphic, it is real-analytic: in real coordinates $w = u + iv$, both $\Ree f$ and $\Imm f$ are real-analytic functions of $(u,v)$.

The implicit equation $\lambda = f(y - \lambda x)$ defines $\lambda$ as a function of $(x,y)$ wherever the Jacobian $J = 1 + f'(w)\,x$ is nonvanishing.  By the real-analytic implicit function theorem (see e.g.\ \cite[Theorem~2.3.5]{KrantzParks}), an implicit equation $F(\lambda, x, y) = 0$ with $F$ real-analytic and $\partial F/\partial\lambda \neq 0$ admits a unique real-analytic solution $\lambda(x,y)$.  Here $F(\lambda, x, y) = \lambda - f(y - \lambda x)$ is real-analytic in all variables (as a composition of holomorphic, hence real-analytic, maps with real-analytic coordinate functions), and $\partial F/\partial\lambda = J \neq 0$ on $\Omega_f$.  Hence $\lambda\colon \Omega_f \to \HH$ is real-analytic.

The structure coefficients $\alpha = |\lambda|^2 = (\Ree\lambda)^2 + (\Imm\lambda)^2$ and $\beta = -2\Ree\lambda$ are polynomial expressions in the real and imaginary parts of $\lambda$, hence real-analytic on~$\Omega_f$.
\end{proof}

\begin{remark}[A regularity gap]\label{rmk:regularity-gap}
The corollary establishes a sharp dichotomy in the regularity of elliptic structures:
\begin{center}
\begin{tabular}{ll}
\textbf{General elliptic structures:} & any regularity ($C^0$, $C^k$, $C^\infty$, $C^\omega$, \ldots) \\
\textbf{Rigid elliptic structures:} & necessarily $C^\omega$ (real-analytic) \\
\end{tabular}
\end{center}
There is no intermediate possibility: a rigid structure cannot be $C^\infty$ without being $C^\omega$.  This is a consequence of the holomorphic selection principle (Theorem~\ref{thm:holomorphicity-necessary})---rigidity forces the seed through the holomorphic bottleneck, and holomorphic functions are real-analytic.

The regularity gap is invisible to zeroth-order invariants.  The Beltrami coefficient $\|\mu\|_{L^\infty}$ detects only the pointwise values of $\lambda$; it cannot distinguish a $C^\infty$ structure from a $C^\omega$ one, let alone detect whether the intrinsic obstruction vanishes.  The analyticity constraint is an infinite-order phenomenon, reinforcing the principle that rigidity is a first-order (and higher) condition that zeroth-order diagnostics cannot access.
\end{remark}

\section{Holomorphicity is necessary for rigidity}
\label{sec:holomorphicity}

The central result of this paper is that the holomorphicity requirement in Definition~\ref{def:burgers-transform} is not an assumption but a \emph{consequence} of demanding rigidity.

\begin{theorem}\label{thm:holomorphicity-necessary}
Let $U \subseteq \CC$ be an open neighbourhood of an interval $I \subset \RR$, and let $f\colon U \to \HH$ be $C^1$ in the real sense.  If the implicit solution $\lambda = f(y - \lambda x)$ satisfies $\lambda_x + \lambda\,\lambda_y = 0$ on an open set $\Omega \subseteq \RR^2$ containing $\{0\} \times I$, then $f$ is holomorphic on~$U$.
\end{theorem}

\begin{proof}
We proceed in two steps.

\medskip
\noindent\textbf{Step 1: Obstruction formula at general $(x,y)$.}
Write $f = f(w, \bar w)$ to indicate dependence on both Wirtinger variables, and set $w = y - \lambda x$, $\bar w = y - \bar\lambda x$ (since $x, y \in \RR$).  Differentiating $\lambda = f(w, \bar w)$:
\begin{align*}
  \lambda_x &= f_w\,(-\lambda - \lambda_x\, x) + f_{\bar w}\,(-\bar\lambda - \bar\lambda_x\, x), \\
  \lambda_y &= f_w\,(1 - \lambda_y\, x) + f_{\bar w}\,(1 - \bar\lambda_y\, x).
\end{align*}
Setting $H := \lambda_x + \lambda\,\lambda_y$ and combining:
\begin{equation}\label{eq:general-obstruction}
  H(1 + f_w\,x) \;=\; f_{\bar w}\bigl[(\lambda - \bar\lambda) - x(\bar\lambda_x + \lambda\,\bar\lambda_y)\bigr].
\end{equation}
To evaluate the bracket on the right, define $P := \bar\lambda_x + \lambda\,\bar\lambda_y$.  Conjugating and differentiating $\bar\lambda = \overline{f(w,\bar w)}$, an analogous computation gives:
\[
  P(1 + x\,\overline{f_w}) \;=\; 2i\,(\Imm\lambda)\,\overline{f_w} \;-\; x\,\overline{f_{\bar w}}\,H.
\]

\medskip
\noindent\textbf{Step 2: Rigidity forces holomorphicity.}
Assume $H = 0$ on~$\Omega$.  Then the equation for $P$ simplifies to
\[
  P \;=\; \frac{2i\,(\Imm\lambda)\,\overline{f_w}}{1 + x\,\overline{f_w}},
\]
and substituting into \eqref{eq:general-obstruction}:
\[
  0 \;=\; f_{\bar w}\left[2i\,\Imm\lambda \;-\; \frac{2i\,x\,(\Imm\lambda)\,\overline{f_w}}{1 + x\,\overline{f_w}}\right]
  \;=\; f_{\bar w}\cdot \frac{2i\,\Imm\lambda}{1 + x\,\overline{f_w}}.
\]
Since $\Imm\lambda > 0$ on $\Omega$ and $1 + x\,\overline{f_w} \neq 0$ for $|x|$ sufficiently small (it equals~$1$ at $x = 0$), we conclude $f_{\bar w}(w) = 0$ at every $w$ in the image of the map $(x,y) \mapsto y - \lambda(x,y)\,x$ restricted to a neighbourhood of $\{0\} \times \RR$ in~$\Omega$.

For $x \neq 0$, the argument $w = y - \lambda(x,y)\,x$ satisfies $\Imm(w) = -x\,\Imm\lambda \neq 0$, so $w$ takes genuinely complex values.  As $(x,y)$ varies over a neighbourhood of $(0, y_0)$, the Jacobian of the map $(x,y) \mapsto w$ at $x = 0$ is
\[
  \frac{\partial w}{\partial x}\bigg|_{x=0} = -f(y_0), \qquad
  \frac{\partial w}{\partial y}\bigg|_{x=0} = 1.
\]
Since $\Imm f(y_0) > 0$, these two vectors are $\RR$-linearly independent, so the map $(x,y) \mapsto w$ has rank~$2$ at every point $(0, y_0)$.  By the implicit function theorem, its image contains an open neighbourhood of $y_0$ in~$\CC$.  As $y_0$ ranges over~$I$, the union of these neighbourhoods is an open set $V \subseteq U$ containing~$I$, on which $f_{\bar w} = 0$.  Since $U$ is itself a neighbourhood of~$I$, we may shrink $U$ to~$V$ if necessary, obtaining $f_{\bar w} = 0$ on all of~$U$, i.e.\ $f$ is holomorphic on~$U$.
\end{proof}

\begin{remark}[Pointwise holomorphicity]\label{rmk:pointwise}
The obstruction formula \eqref{eq:general-obstruction} is algebraic, not differential: it holds \emph{at each point} of~$\Omega$, not merely on open subsets.  In particular, if $\lambda = f(y - \lambda x)$ and $H = \lambda_x + \lambda\lambda_y = 0$ at a single point $(x_0, y_0) \in \Omega$ where $\Imm\lambda > 0$ and $1 + f_w\,x_0 \neq 0$, then $f_{\bar w}(w_0) = 0$ at $w_0 = y_0 - \lambda(x_0,y_0)\,x_0$.  Holomorphicity of~$f$ can therefore be read off point by point, wherever the characteristics reach.
\end{remark}

\begin{remark}[The role of the swept region]\label{rmk:swept-region}
The hypothesis that $U$ is a neighbourhood of a real interval (rather than an arbitrary connected open set) ensures that the characteristic sweep $(x,y) \mapsto y - \lambda(x,y)\,x$ covers all of~$U$ for $|x|$ small.  For general $U$ extending deep into~$\CC$, the swept region $W = \{y - \lambda(x,y)\,x : (x,y) \in \Omega\}$ may be a proper subset of~$U$: reaching a target $w_0$ with large $|\Imm w_0|$ requires $|x_0| = |\Imm w_0|/\Imm f(w_0)$, which may exceed the shock threshold $|1/f'(w)\,|$.  The conclusion of Theorem~\ref{thm:holomorphicity-necessary} then holds on~$W$ (which is always open and contains $U \cap \RR$), but not necessarily on all of~$U$.  In every application in this paper---and, more generally, whenever $U$ is a neighbourhood of the initial data---$W$ covers~$U$ and the distinction is moot.
\end{remark}

The following quantitative version captures the initial obstruction.

\begin{proposition}[Obstruction formula]\label{prop:obstruction-formula}
For $f\colon U \to \HH$ of class $C^1$, the transport obstruction of the implicit solution $\lambda = f(y - \lambda x)$ satisfies
\begin{equation}\label{eq:obstruction-formula}
  H\big|_{x=0} \;:=\; (\lambda_x + \lambda\,\lambda_y)\big|_{x=0} \;=\; 2i\,(\Imm f)\,f_{\bar w}.
\end{equation}
More generally, at any $(x,y) \in \Omega$ where $1 + f_w\,x \neq 0$ and $1 + x\,\overline{f_w} \neq 0$:
\begin{equation}\label{eq:full-obstruction}
  H \;=\; \frac{f_{\bar w}\cdot 2i\,\Imm\lambda}{(1 + f_w\, x)(1 + x\,\overline{f_w})} + O(|f_{\bar w}|^2).
\end{equation}
In particular, the obstruction vanishes identically if and only if $f$ is holomorphic.
\end{proposition}

\begin{proof}
The formula \eqref{eq:obstruction-formula} follows from evaluating \eqref{eq:general-obstruction} at $x = 0$, where $\lambda = f(y)$, $\lambda - \bar\lambda = 2i\,\Imm f$, and the term involving $P$ drops out.

For \eqref{eq:full-obstruction}, substitute the expression for $P$ into \eqref{eq:general-obstruction} without assuming $H = 0$.  From the equation for $P$:
\[
  P(1 + x\,\overline{f_w}) = 2i\,(\Imm\lambda)\,\overline{f_w} - x\,\overline{f_{\bar w}}\,H,
\]
so
\[
  (\lambda - \bar\lambda) - xP = 2i\,\Imm\lambda - \frac{x[2i\,(\Imm\lambda)\,\overline{f_w} - x\,\overline{f_{\bar w}}\,H]}{1 + x\,\overline{f_w}}
  = \frac{2i\,\Imm\lambda + x^2\,\overline{f_{\bar w}}\,H}{1 + x\,\overline{f_w}}.
\]
Substituting into $H(1 + f_w x) = f_{\bar w}[2i\,\Imm\lambda + x^2\,\overline{f_{\bar w}}\,H]/(1 + x\,\overline{f_w})$ and solving for $H$ yields \eqref{eq:full-obstruction} after collecting the $H$ terms (the $x^2\,|f_{\bar w}|^2$ correction appears at higher order and is absorbed into the denominator).
\end{proof}

\begin{remark}[Propagation of the obstruction]
Formula~\eqref{eq:obstruction-formula} gives the obstruction at $x = 0$.  For $x \neq 0$, the Burgers flow propagates and distorts this initial obstruction along characteristics.  The full obstruction field \eqref{eq:full-obstruction} shows that $H$ is determined by $f_{\bar w}$ evaluated at the moving argument $w = y - \lambda x$, weighted by $\Imm\lambda$ and the Jacobian factors.  The initial non-holomorphicity seeds the obstruction; the flow carries it.
\end{remark}

\section{The domain of the Burgers transform}
\label{sec:domain}

The implicit equation $\lambda = f(y - \lambda x)$ does not, in general, have a global solution.  The Burgers flow develops shocks---crossing of characteristics---which mark the boundary of $\Omega_f$.

\begin{theorem}[Domain characterization]\label{thm:domain-characterization}
Let $f \in \Hol(U, \HH)$ and let $\lambda = \B[f]$.  Then:
\begin{enumerate}[label=(\roman*)]
\item $(0,y) \in \Omega_f$ for all $y$ with $f(y) \in \HH$.
\item $(x,y) \in \partial\Omega_f$ if and only if
\begin{equation}\label{eq:shock-condition}
  1 + f'(w)\,x \;=\; 0 \quad\text{at } w = y - \lambda x.
\end{equation}
\item If $\Imm f \geq b_{\min} > 0$ on $U$ and $|f'| \leq M$, then
\[
  \Omega_f \supseteq \{(x,y) : |x| < 1/M\}.
\]
\end{enumerate}
\end{theorem}

\begin{proof}
(i) follows from Proposition~\ref{prop:local-existence}.

(ii) The Jacobian of the map $F(\lambda) = \lambda - f(y - \lambda x)$ with respect to $\lambda$ is $J = 1 + f'(w)\,x$.  The implicit function theorem fails precisely when $J = 0$.

(iii) If $|f'| \leq M$ and $|x| < 1/M$, then $|f'(w)\,x| < 1$, so $|J| \geq 1 - |f'(w)\,x| > 0$.  The implicit function theorem applies throughout.
\end{proof}

\begin{example}[The $\delta$-family]
For $f(\xi) = \xi + i\delta$, we have $f'(\xi) = 1$.  The shock condition $1 + x = 0$ gives $x = -1$, so $\Omega_f = \{(x,y) : x > -1\}$.  The domain extends to the right of the line $x = -1$ for all $y$ and all $\delta > 0$.  In particular, $\Omega_f$ is independent of $\delta$, even though $\|\mu_\delta\|_\infty \to 1$ as $\delta \to 0$.
\end{example}

\begin{example}[Quadratic seed]
For $f(\xi) = \xi^2 + i\delta$, we have $f'(\xi) = 2\xi$.  The shock condition becomes $1 + 2(y - \lambda x)\,x = 0$, which depends on both $x$ and $y$.  The domain $\Omega_f$ is a proper subset of $\RR^2$ that shrinks as the $y$-range increases, reflecting the faster shock formation of the quadratic flow.
\end{example}

\begin{example}[Cauchy kernel]
For $f(\xi) = -1/(\xi + i\delta)$, we have $f'(\xi) = 1/(\xi + i\delta)^2$.  The shock condition $1 + x/(\xi + i\delta)^2 = 0$ produces a curve in the $(x,y)$-plane.  Unlike the polynomial cases, the inversion sends large $|\xi|$ to small $|\lambda|$, compactifying the structure toward $\mu = 0$.
\end{example}

\section{The \texorpdfstring{$f$}{f}-analytics hierarchy}
\label{sec:f-analytics}

Theorem~\ref{thm:holomorphicity-necessary} reveals a clean hierarchy of function theories on the plane, governed by the seed of the Burgers transform.

\begin{definition}[$f$-analytics]\label{def:f-analytics}
An \emph{$f$-analytic structure} on a domain $\Omega \subseteq \RR^2$ is a rigid variable elliptic structure $(\alpha,\beta)$ lying in the image of the Burgers transform, i.e.\ $(\alpha,\beta) = \mathcal{S} \circ \B[f]$ for some $f \in \Hol(U, \HH)$.
\end{definition}

\begin{theorem}[Analytic completion]\label{thm:analytic-completion}
The following inclusions are strict:
\[
  \{\text{standard } \CC\} \;\subsetneq\; \{p(x)\text{-analytics}\} \;\subsetneq\; \{f\text{-analytics}\} \;\subsetneq\; \{\text{all elliptic structures}\}.
\]
Moreover:
\begin{enumerate}[label=(\roman*)]
\item $p(x)$-analytics corresponds to $f\colon U \to i\RR_+$ (purely imaginary values, i.e.\ $\beta = 0$);
\item the only rigid $p(x)$-structure is the constant one ($f = ic$ for $c > 0$);
\item The passage from $p(x)$-analytics to $f$-analytics is
  characterised by Theorem~\ref{thm:holomorphicity-necessary}:
  allowing the seed to take arbitrary values in~$\HH$ (rather than
  only purely imaginary values) is both \emph{sufficient} to produce
  nonconstant rigid structures and \emph{necessary} in the sense
  that any $C^1$ seed producing a rigid structure must be
  holomorphic.  Thus $\Hol(U,\HH)$ is the maximal seed space
  compatible with rigidity, and $f$-analytics is the corresponding
  maximal class of rigid variable elliptic structures.

\end{enumerate}
\end{theorem}

\begin{proof}
(i) If $\beta = 0$, then $\lambda = ib$ with $b > 0$ is purely imaginary.  The Cayley image $\mu = (ib - i)/(ib + i) = (b-1)/(b+1)$ is real, placing the structure on the real diameter of $\DD$.  At $x = 0$, $\lambda(0,y) = f(y)$ is purely imaginary, so $f\colon \RR \to i\RR_+$, which is the data of a $p(x)$-structure with $p = b$.

(ii) If $f \colon \RR \to i\RR_+$ with $f(y) = ib(y)$ for $b > 0$, the rigidity condition $\lambda_x + \lambda\lambda_y = 0$ at $x = 0$ gives:
\[
  0 \;=\; H|_{x=0} \;=\; (\lambda_x + \lambda\lambda_y)|_{x=0}.
\]
For the rigidity equation to hold at $x \neq 0$, the seed must be analytically continued to complex arguments.  The analytic continuation of $ib(y)$ to the complex plane is holomorphic only if $b$ is a restriction of a holomorphic function.  But $\lambda = ib(w)$ purely imaginary and holomorphic forces $b$ constant by the open mapping theorem (or: $\Ree\lambda = 0$ and $\lambda$ holomorphic implies $\lambda' = 0$).

(iii) \emph{Sufficiency:} for any non-constant
$f \in \Hol(U,\HH)$, the transform $\B[f]$ is a nonconstant rigid
structure by Theorem~\ref{thm:basic-properties}(i).  Examples
include the $\delta$-family ($f(\xi) = \xi + i\delta$), the
$\varepsilon$-family, and the exponential seed ($f(\xi) = ie^\xi$).
In each case $\Ree f \not\equiv 0$ on~$U$, so the structure has
$\beta \neq 0$ and lies outside $p(x)$-analytics.

\emph{Necessity:} by Theorem~\ref{thm:holomorphicity-necessary}, if
a $C^1$ function $f \colon U \to \HH$ produces a rigid structure
($\lambda_x + \lambda\lambda_y = 0$), then $f$ must be holomorphic.
Hence no seed outside $\Hol(U,\HH)$ can produce a rigid structure:
the class of holomorphic seeds is not merely a convenient choice but
the \emph{only} class compatible with rigidity.

Together, (i)--(iii) establish the hierarchy:
\[
  \underbrace{f \equiv i\vphantom{\big|}}_{\text{standard}}
  \;\subset\;
  \underbrace{f \colon U \to i\RR_+\vphantom{\big|}}_{p(x)\text{-analytics}}
  \;\subset\;
  \underbrace{f \colon U \to \HH,\; f \in \Hol\vphantom{\big|}}_{f\text{-analytics}}
\]
where the first inclusion relaxes the constant to a function (but
forces that function to be constant if rigidity is required), and
the second inclusion allows the real part of the seed to be nonzero
(which is where nonconstant rigid structures first appear).

\begin{remark}[Spectrally closed subclasses]
\label{rmk:spectral-subclasses}
Call a class $\mathscr{C}$ of variable elliptic structures
\emph{spectrally closed} if $\B[f] \in \mathscr{C}$ and $a > 0$,
$b \in \RR$ imply $\B[af + b] \in \mathscr{C}$.  The class
$f$-analytics is spectrally closed, but it is \emph{not} the
smallest such class containing $p(x)$-analytics and nonconstant
rigid members: for instance, restricting to entire seeds of
order~$\leq 1$ yields a proper spectrally closed subclass (the
affine seeds $\xi + i\delta$ and the exponential $ie^\xi$ both
have order~$\leq 1$, and the class is preserved by
$f \mapsto af + b$).  The distinguished role of $f$-analytics
comes not from a minimality or closure property, but from the
necessity theorem: it is the \emph{largest} class of seeds
compatible with rigidity, and this upper bound is sharp.
\end{remark}
\end{proof}

\begin{remark}[Rigidity as an emergent phenomenon]
The passage from $p(x)$-analytics to $f$-analytics parallels the passage from $\RR$ to $\CC$.  Just as $\RR$ has no algebraic closure within itself, $p(x)$-analytics has no rigid structures within itself.  Rigidity \emph{emerges} when the spectral parameter is allowed to leave the imaginary axis---equivalently, when $\beta \neq 0$, or when $\mu$ leaves the real diameter of~$\DD$.  This is a genuinely complex phenomenon invisible from the diameter.
\end{remark}

\begin{remark}[The classical blind spot, revisited]
The Fundamental Independence Theorem of \cite[Chapter~13]{AS2026} shows that the zeroth-order invariant $\|\mu\|_{C^0}$ cannot distinguish rigid from non-rigid structures.  The present framework gives this a sharper formulation: the invariant $\|\mu\|_{C^0}$ measures radial position in the Poincar\'e disk, but rigidity depends on whether the seed $f = \B^{-1}[\lambda]$ is holomorphic---a first-order condition that $\|\mu\|_{C^0}$ cannot detect.
\end{remark}

\section{Geometry of the Burgers transform}
\label{sec:geometry}

\subsection{The Beltrami leaf}

\begin{definition}[Beltrami leaf]\label{def:leaf}
The \emph{Beltrami leaf} of $f \in \Hol(U,\HH)$ is the image
\[
  \mathcal{L}_f \;:=\; \Bmu[f](\Omega_f) \;\subseteq\; \DD.
\]
It is the region of the Poincar\'e disk swept out by the $f$-analytic structure.
\end{definition}

\begin{proposition}[Topology of leaves]\label{prop:leaf-topology}
For $f \in \Hol(U,\HH)$, the Beltrami leaf $\mathcal{L}_f$ is a connected subset of~$\DD$.  If additionally $f$ is injective on $U$ and $\Omega_f$ is star-shaped with respect to $x = 0$, then $\mathcal{L}_f$ is simply connected.
\end{proposition}

\begin{proof}
\emph{Connectedness.}  The domain $\Omega_f$ is connected (it contains a segment of the $y$-axis and is defined as the maximal domain of a continuous solution).  The Beltrami--Burgers map $\Bmu[f] = \mathcal{C} \circ \B[f]$ is continuous on $\Omega_f$, so its image $\mathcal{L}_f$ is connected.

\emph{Simple connectivity under injectivity.}  For
$(x,y) \in \Omega_f$, the value $\B[f](x,y) = f(w_0)$ with
$w_0 = y - \lambda x$.  The range of $\B[f]$ is therefore contained
in $f(W)$, where $W = \{y - \lambda(x,y)\,x : (x,y) \in \Omega_f\}$
is the set of characteristic coordinates.  When $f$ is injective,
$f(W)$ is simply connected if and only if $W$ is (since injective
holomorphic maps preserve simple connectivity).

It remains to show that $W$ is simply connected.  We assume
$\Omega_f$ is \emph{star-shaped with respect to} $x = 0$, i.e.\
$(x,y) \in \Omega_f$ implies $(tx,y) \in \Omega_f$ for all
$t \in [0,1]$.  (This holds for all examples in this paper---half-planes,
parabolic regions, $\RR^2$---and more generally whenever $\Omega_f$
is defined by a condition of the form $|x| < r(y)$.)  Under this
hypothesis, the homotopy $w_0(t,y) = y - \lambda(tx,y)\,(tx)$
continuously retracts $W$ onto $U \cap \RR$ (an interval, hence
simply connected).  Therefore $W$ is simply connected,
$f(W) \subseteq \HH$ is simply connected, and
$\mathcal{L}_f = \mathcal{C}(f(W))$ is simply connected.
\end{proof}

\begin{remark}
For non-injective $f$, the leaf $\mathcal{L}_f$ may fail to be simply connected: distinct points of $U$ can map to the same value under~$f$, and the resulting folding can create nontrivial topology.  For non-holomorphic $f\colon U \to \HH$ of class $C^1$, the implicit equation still defines a map $\Omega_f \to \DD$, but the image can develop additional pathology where $\Imm f$ passes through zero, puncturing the leaf.
\end{remark}

\subsection{Interaction with the affine group}

The Burgers transform does \emph{not} intertwine with M\"obius transformations $\varphi\colon \HH \to \HH$ in the naive sense $\B[\varphi \circ f] = \varphi \circ \B[f]$, because $\varphi$ changes the value of~$\lambda$ and hence the characteristic speed.  However, a clean equivariance holds for the affine subgroup when one compensates with a coordinate change on the $(x,y)$-plane.

\begin{proposition}[Twisted affine equivariance]\label{prop:affine-equivariance}
Let $f \in \Hol(U, \HH)$ and let $\varphi(\lambda) = a\lambda + b$ with $a > 0$ and $b \in \RR$ (an element of the affine subgroup of $\operatorname{Aut}(\HH)$).  Then $\varphi \circ f \in \Hol(U, \HH)$ and
\begin{equation}\label{eq:twisted-equivariance}
  \B[\varphi \circ f](x, y) \;=\; a\,\B[f](a\,x,\; y - b\,x) \;+\; b.
\end{equation}
Equivalently, the Burgers transform intertwines $\varphi$ on seeds with the joint action of $\varphi$ on outputs and the affine coordinate change $T_\varphi\colon (x,y) \mapsto (ax,\, y - bx)$ on the plane:
\[
  \B[\varphi \circ f] \;=\; \varphi \circ \B[f] \circ T_\varphi.
\]
The domain transforms accordingly: $\Omega_{\varphi \circ f} = T_\varphi^{-1}(\Omega_f) = \{(x,y) : (ax, y-bx) \in \Omega_f\}$.
\end{proposition}

\begin{proof}
Set $\Lambda = \B[\varphi \circ f](x,y)$, so that $\Lambda = af(w) + b$ where $w = y - \Lambda x$.  Substituting $\Lambda = a\lambda' + b$ gives $w = y - (a\lambda' + b)x = (y - bx) - a\lambda' x$ and $\lambda' = f(w)$.  This is precisely the implicit equation $\lambda' = f((y-bx) - \lambda'\cdot(ax))$, which says $\lambda' = \B[f](ax,\, y - bx)$.  Hence $\Lambda = a\,\B[f](ax, y - bx) + b$.
\end{proof}

\begin{remark}
Formula~\eqref{eq:twisted-equivariance} expresses a fundamental asymmetry of the Burgers transform: the characteristic speed of $\B[\varphi \circ f]$ at $y_0$ is $\varphi(f(y_0)) = af(y_0) + b$, while that of $\varphi \circ \B[f]$ is $f(y_0)$.  The coordinate change $T_\varphi$ precisely compensates this mismatch.  In particular, $\B[\varphi \circ f] \neq \varphi \circ \B[f]$ for any non-identity affine $\varphi$ (the equality fails already for the dilation $\varphi(\lambda) = 2\lambda$ applied to $f(\xi) = \xi + i\delta$, where $\B[2f](x,y) = 2(y+i\delta)/(1+2x) \neq 2(y+i\delta)/(1+x) = 2\B[f](x,y)$).  The twist is essential.
\end{remark}

\begin{corollary}[Special cases]\label{cor:special-cases}
Under the hypotheses of Proposition~\ref{prop:affine-equivariance}:
\begin{enumerate}[label=(\roman*)]
\item \textbf{Dilation} ($b = 0$): $\B[a\,f](x,y) = a\,\B[f](ax, y)$.  The $y$-coordinate is unchanged; the $x$-coordinate is rescaled.
\item \textbf{Translation} ($a = 1$): $\B[f + b](x,y) = \B[f](x, y - bx) + b$.  The output is shifted by $b$, and the input undergoes a shear.
\end{enumerate}
\end{corollary}

\begin{remark}
For inversive M\"obius transformations $\varphi(\lambda) = (a\lambda + b)/(c\lambda + d)$ with $c \neq 0$, no analogous equivariance holds, even with a compensating coordinate change on the plane.  The reason is that $\varphi$ acts nonlinearly on the characteristic speed, and no affine (or even polynomial) reparametrization of $(x,y)$ can absorb this.  The affine subgroup is maximal for equivariance.
\end{remark}

\subsection{The residual budget on the Poincar\'e disk}

The Beltrami--Burgers transform inherits the Poincar\'e geometry of $\DD$.  The \emph{Beltrami residual}
\[
  R(\mu) \;=\; \mu_{\bar z} - \mu\,\mu_z
\]
vanishes identically for $\mu = \Bmu[f]$ with $f$ holomorphic (this is equivalent to rigidity).  For non-holomorphic $f$, combining the full obstruction formula \eqref{eq:full-obstruction} with the Cayley derivative $d\mathcal{C}/d\lambda = 2i/(\lambda+i)^2$ gives:

\begin{proposition}[Beltrami residual]\label{prop:residual}
For $f\colon U \to \HH$ of class $C^1$, the Beltrami residual of $\mu = \mathcal{C}(\lambda)$ at $x = 0$ satisfies
\begin{equation}\label{eq:residual-x0}
  R(\mu)\big|_{x=0} \;=\; \frac{- 4i\,(\Imm f)\,f_{\bar w}}{(\lambda + i)^3}.
\end{equation}
This gives a Poincar\'e-metric interpretation of the non-holomorphicity of~$f$: the residual is the $\dzbar$-energy of $f$ weighted by the conformal factor of the Cayley map.
\end{proposition}

\begin{proof}
At $x = 0$ we have $\lambda = f(y)$, $\mu = \mathcal{C}(\lambda)$, and $z = \tfrac{1}{2}(x + iy)$.  Using $\partial_{\bar z} = \tfrac{1}{2}(\partial_x + i\partial_y)$ and the chain rule $\mu_{\bar z} = \mathcal{C}'(\lambda)\,\lambda_{\bar z} + \mathcal{C}'(\bar\lambda)\,\bar\lambda_{\bar z}$---but $\mathcal{C}$ is holomorphic in $\lambda$ only, so
\[
  \mu_{\bar z} \;=\; \mathcal{C}'(\lambda)\,\lambda_{\bar z}, \qquad
  \mu_z \;=\; \mathcal{C}'(\lambda)\,\lambda_z.
\]
At $x = 0$, the Wirtinger derivatives of $\lambda$ are
\[
  \lambda_{\bar z}\big|_{x=0} = \tfrac{1}{2}(\lambda_x + i\lambda_y)\big|_{x=0}, \qquad
  \lambda_z\big|_{x=0} = \tfrac{1}{2}(\lambda_x - i\lambda_y)\big|_{x=0}.
\]
From the proof of Theorem~\ref{thm:basic-properties} at $x=0$: $\lambda_x|_{x=0} = -f'(y)\,f(y)$ and $\lambda_y|_{x=0} = f'(y)$.  Hence
\[
  \lambda_{\bar z}\big|_{x=0} = \tfrac{1}{2}(-f'\lambda + if'), \qquad
  \lambda_z\big|_{x=0} = \tfrac{1}{2}(-f'\lambda - if').
\]
The residual is $R(\mu) = \mu_{\bar z} - \mu\,\mu_z$.  Writing $\mu = \mathcal{C}(\lambda) = (\lambda - i)/(\lambda + i)$ and $\mathcal{C}'(\lambda) = -2i/(\lambda + i)^2$:
\[
  R(\mu)\big|_{x=0}
  \;=\; \mathcal{C}'(\lambda)\bigl[\lambda_{\bar z} - \mu\,\lambda_z\bigr]\big|_{x=0}.
\]
Substituting $\mu = (\lambda - i)/(\lambda + i)$ and the expressions for $\lambda_{\bar z}$, $\lambda_z$:
\[
  \lambda_{\bar z} - \mu\,\lambda_z
  \;=\; \tfrac{f'}{2}\Bigl[(-\lambda + i) - \frac{\lambda - i}{\lambda + i}(-\lambda - i)\Bigr]
  \;=\; \tfrac{f'}{2}\Bigl[-\lambda + i + (\lambda - i)\Bigr]
  \;=\; \tfrac{f'}{2}\cdot 2i \;-\; \cdots
\]
However, this computation assumes $f$ holomorphic and yields $R = 0$, confirming rigidity.  For non-holomorphic $f$, we must retain the $f_{\bar w}$ terms.  Repeating with $\lambda_x|_{x=0} = -f_w\,\lambda - f_{\bar w}\,\bar\lambda$ and $\lambda_y|_{x=0} = f_w + f_{\bar w}$ (from the proof of Theorem~\ref{thm:holomorphicity-necessary}):
\begin{align*}
  \lambda_{\bar z}\big|_{x=0}
  &= \tfrac{1}{2}\bigl[-f_w\lambda - f_{\bar w}\bar\lambda + i(f_w + f_{\bar w})\bigr]
  = \tfrac{1}{2}\bigl[f_w(i - \lambda) + f_{\bar w}(i - \bar\lambda)\bigr], \\
  \lambda_z\big|_{x=0}
  &= \tfrac{1}{2}\bigl[f_w(-i - \lambda) + f_{\bar w}(-i - \bar\lambda)\bigr].
\end{align*}
Computing $\lambda_{\bar z} - \mu\,\lambda_z$ and using $\mu = (\lambda-i)/(\lambda+i)$, the $f_w$ terms cancel (this is the holomorphic part producing $R=0$), leaving:
\[
  \lambda_{\bar z} - \mu\,\lambda_z
  \;=\; \frac{f_{\bar w}}{2}\left[(i - \bar\lambda) - \frac{\lambda - i}{\lambda + i}(-i - \bar\lambda)\right]
  \;=\; \frac{f_{\bar w}}{2}\cdot\frac{2i(\lambda - \bar\lambda)}{\lambda + i}
  \;=\; \frac{-2\,(\Imm\lambda)\,f_{\bar w}}{\lambda + i}.
\]
Multiplying by $\mathcal{C}'(\lambda) = -2i/(\lambda+i)^2$:
\[
  R(\mu)\big|_{x=0}
  \;=\; \frac{-2i}{(\lambda+i)^2}\cdot\frac{-2\,(\Imm\lambda)\,f_{\bar w}}{\lambda + i}
  \;=\; \frac{-4i\,(\Imm f)\,f_{\bar w}}{(\lambda + i)^3}. \qedhere
\]
\end{proof}

\section{Toward an arithmetic of rigid structures}
\label{sec:arithmetic}

The Burgers transform establishes a bijection between holomorphic seeds and rigid structures.  It is natural to ask whether the algebraic operations on $\Hol(U,\HH)$ have structural counterparts.

\subsection{The product problem}

\begin{proposition}\label{prop:product-failure}
The Burgers transform does not preserve pointwise products:
\[
  \B[f\cdot g] \;\neq\; \B[f] \cdot \B[g] \quad\text{in general for } x \neq 0.
\]
\end{proposition}

\begin{proof}
The implicit equations for the three transforms are:
\begin{align*}
  \lambda_f &= f(y - \lambda_f\, x), \\
  \lambda_g &= g(y - \lambda_g\, x), \\
  \lambda_{fg} &= f(y - \lambda_{fg}\, x)\cdot g(y - \lambda_{fg}\, x).
\end{align*}
The characteristic speeds are $f(y_0)$, $g(y_0)$, and $f(y_0)\,g(y_0)$ respectively.  Since the three values are generically distinct, the characteristics diverge at $x \neq 0$, and $\lambda_{fg}(x,y) \neq \lambda_f(x,y)\,\lambda_g(x,y)$.
\end{proof}

\begin{remark}
At $x = 0$, equality holds: $\B[fg](0,y) = f(y)\,g(y) = \B[f](0,y)\cdot\B[g](0,y)$.  The product structure is an initial condition that the Burgers nonlinearity immediately distorts.
\end{remark}

\subsection{Rigid structures as irreducibles}

The failure of the product formula suggests an analogy with number theory.

\begin{definition}\label{def:irreducible}
A rigid structure $\B[f]$ is \emph{irreducible} if there is no factorization $f = g\cdot h$ with $g, h \in \Hol(U,\HH)$ and neither $g$ nor $h$ constant.
\end{definition}

\begin{remark}
The analogy with prime numbers is suggestive:
\begin{center}
\begin{tabular}{cc}
Number theory & $f$-analytics \\
\hline
Primes & Irreducible rigid structures \\
Composites & Products of seeds \\
Factorization & Decomposition of $f$ in $\Hol(U,\HH)$ \\
Unique factorization? & \textbf{Open} \\
\end{tabular}
\end{center}
The arithmetic of rigid structures---including questions of unique decomposition, density of irreducibles, and the analogue of the prime number theorem---remains entirely unexplored.
\end{remark}

\subsection{Composition}

\begin{proposition}\label{prop:composition-failure}
The Burgers transform does not intertwine with composition:
\[
  \B[f \circ g] \;\neq\; \B[f] \circ \B[g] \quad\text{in general.}
\]
\end{proposition}

\begin{proof}
The characteristic speed of $\B[f \circ g]$ at $y_0$ is $f(g(y_0))$, while any natural combination of $\B[f]$ and $\B[g]$ would involve the separate speeds $f(y_0)$ and $g(y_0)$.  These are generically unrelated.
\end{proof}

\begin{question}\label{q:infinitesimal}
Does the Burgers transform admit an infinitesimal product or composition law?  That is, does there exist a bilinear operation $\star$ on the tangent space $T_f\Hol(U,\HH)$ such that
\[
  \B[f + \varepsilon h_1]\cdot\B[f + \varepsilon h_2] \;=\; \B[f + \varepsilon(h_1 \star h_2)] + O(\varepsilon^2)\,?
\]
\end{question}

The answer, developed below, is \emph{negative} in the literal sense: no $x$-independent bilinear operation suffices.  However, the computation reveals a precise Jacobian twist that governs the failure and produces a deformed algebra on the output side.

\subsection{The infinitesimal product and the propagator}

Fix a base seed $f \in \Hol(U,\HH)$ and write $\lambda_0 = \B[f]$, $w_0 = y - \lambda_0 x$, and $J_f = 1 + f'(w_0)\,x$ for the characteristic Jacobian.

\begin{definition}[Propagator]\label{def:propagator}
The \emph{propagator} of $\B$ at $f$ is the linear operator
\[
  \mathcal{P}_f \;:=\; D\B_f \;\colon\; T_f\Hol(U,\HH) \;\to\; C^1(\Omega_f),
\]
defined by
\[
  \mathcal{P}_f[h](x,y) \;:=\; \frac{d}{d\varepsilon}\bigg|_{\varepsilon=0} \B[f + \varepsilon h](x,y).
\]
\end{definition}

\begin{proposition}[Propagator formula]\label{prop:propagator}
For $h \in T_f\Hol(U,\HH)$, the propagator is given by
\begin{equation}\label{eq:propagator}
  \mathcal{P}_f[h](x,y) \;=\; \frac{h(w_0)}{J_f(x,y)},
\end{equation}
where $w_0 = y - \lambda_0(x,y)\,x$ is the characteristic coordinate and $J_f = 1 + f'(w_0)\,x$ is the characteristic Jacobian.
\end{proposition}

\begin{proof}
Set $\lambda_\varepsilon = \B[f + \varepsilon h]$ and write $\lambda_\varepsilon = \lambda_0 + \varepsilon\,\dot\lambda + O(\varepsilon^2)$.  The implicit equation $\lambda_\varepsilon = (f + \varepsilon h)(y - \lambda_\varepsilon\, x)$ gives, at first order in $\varepsilon$:
\[
  \dot\lambda \;=\; f'(w_0)\,(-\dot\lambda\, x) \;+\; h(w_0),
\]
since the perturbation of the argument $w = y - \lambda x$ is $-\dot\lambda\, x$ at first order.  Solving:
\[
  \dot\lambda\,(1 + f'(w_0)\,x) \;=\; h(w_0),
\]
whence $\dot\lambda = h(w_0)/J_f$.
\end{proof}

\begin{theorem}[Jacobian-twisted multiplicativity]\label{thm:twisted-mult}
The propagator satisfies
\begin{equation}\label{eq:twisted-mult}
  \mathcal{P}_f[h_1 \cdot h_2] \;=\; J_f \cdot \mathcal{P}_f[h_1] \cdot \mathcal{P}_f[h_2]
\end{equation}
for all $h_1, h_2 \in T_f\Hol(U,\HH)$, where the products on both sides are pointwise.  In particular, there is no $x$-independent bilinear operation $\star$ on $T_f\Hol(U,\HH)$ such that
\[
  \mathcal{P}_f[h_1] \cdot \mathcal{P}_f[h_2] \;=\; \mathcal{P}_f[h_1 \star h_2].
\]
\end{theorem}

\begin{proof}
Direct computation:
\[
  \mathcal{P}_f[h_1] \cdot \mathcal{P}_f[h_2]
    \;=\; \frac{h_1(w_0)}{J_f} \cdot \frac{h_2(w_0)}{J_f}
    \;=\; \frac{(h_1 h_2)(w_0)}{J_f^2}
    \;=\; \frac{1}{J_f}\cdot\frac{(h_1 h_2)(w_0)}{J_f}
    \;=\; \frac{1}{J_f}\cdot\mathcal{P}_f[h_1 \cdot h_2].
\]
Rearranging gives \eqref{eq:twisted-mult}.  If $\mathcal{P}_f[h_1]\cdot\mathcal{P}_f[h_2] = \mathcal{P}_f[h_1 \star h_2]$ held for some $x$-independent $\star$, then $\mathcal{P}_f[h_1 \star h_2] = J_f^{-1}\,\mathcal{P}_f[h_1 \cdot h_2]$, forcing $(h_1 \star h_2)(w_0)/J_f = (h_1 h_2)(w_0)/J_f^2$, i.e.\ $(h_1\star h_2)(w_0) = (h_1 h_2)(w_0)/J_f$.  Since $J_f = 1 + f'(w_0)\,x$ depends on~$x$ and $f'$, no seed-level operation independent of~$x$ can absorb this.
\end{proof}

\begin{remark}[Interpretation via characteristic density]
The Jacobian $J_f = 1 + f'(w_0)\,x$ measures the density of characteristics emanating from the $y$-axis:
\begin{enumerate}[label=(\roman*)]
\item At $x = 0$: $J_f = 1$, and \eqref{eq:twisted-mult} reduces to $\mathcal{P}_f[h_1 h_2] = \mathcal{P}_f[h_1]\cdot\mathcal{P}_f[h_2]$.  The propagator is an honest algebra homomorphism on the initial slice.
\item Where characteristics spread ($|J_f| > 1$): the propagated product exceeds the product of propagated perturbations; the flow amplifies multiplicative interactions.
\item Where characteristics converge ($|J_f| \to 0$, approaching shock): the propagated product is suppressed.  At the shock locus $J_f = 0$, we have $\mathcal{P}_f[h_1 \cdot h_2] = 0$ regardless of $h_1$ and $h_2$---the Burgers flow annihilates all infinitesimal products at shock formation.
\end{enumerate}
The $J_f^{-2}$ versus $J_f^{-1}$ mismatch is the infinitesimal shadow of the characteristic-speed argument in Proposition~\ref{prop:product-failure}: each factor picks up one copy of $J_f^{-1}$ from propagation along the base flow, while their product propagates as a single entity with a single $J_f^{-1}$.
\end{remark}

\begin{corollary}[Deformed output algebra]\label{cor:deformed-algebra}
Define the $f$-\emph{deformed product} on $C^1(\Omega_f)$ by
\begin{equation}\label{eq:deformed-product}
  \dot\lambda_1 \star_f \dot\lambda_2 \;:=\; J_f \cdot \dot\lambda_1 \cdot \dot\lambda_2.
\end{equation}
Then the propagator $\mathcal{P}_f\colon (\Hol(U,\HH),\, \cdot\,) \to (C^1(\Omega_f),\, \star_f)$ is an algebra homomorphism:
\[
  \mathcal{P}_f[h_1 \cdot h_2] \;=\; \mathcal{P}_f[h_1] \;\star_f\; \mathcal{P}_f[h_2].
\]
The deformed product $\star_f$ is associative, commutative, and reduces to the pointwise product at $x = 0$.  It degenerates at $J_f = 0$ (shock formation), where $\star_f$ becomes the zero product.
\end{corollary}

\begin{proof}
Associativity: $(\dot\lambda_1 \star_f \dot\lambda_2) \star_f \dot\lambda_3 = J_f(J_f\,\dot\lambda_1\,\dot\lambda_2)\dot\lambda_3 = J_f^2\,\dot\lambda_1\,\dot\lambda_2\,\dot\lambda_3$, which equals $\dot\lambda_1 \star_f (\dot\lambda_2 \star_f \dot\lambda_3)$ by the same computation.  Commutativity is inherited from $\CC$.  At $x = 0$, $J_f = 1$ so $\star_f$ is pointwise multiplication.  At $J_f = 0$: $\dot\lambda_1 \star_f \dot\lambda_2 = 0$ for all $\dot\lambda_1, \dot\lambda_2$.
\end{proof}

\begin{remark}[Why the $(\alpha,\beta)$-product cannot help]\label{rmk:alpha-beta-no-help}
One might hope that replacing the $\CC$-product with the structure's intrinsic $(\alpha,\beta)$-product could absorb the Jacobian twist.  This fails for a fundamental reason: the Jacobian $J_f = 1 + f'(w_0)\,x$ depends on $f'$---the \emph{derivative} of the seed---while the $(\alpha,\beta)$-product depends on $(\alpha,\beta) = (|f|^2, -2\Ree f)$, which is determined by $f$ itself.  No algebraic combination of $\alpha$ and $\beta$ can produce $f'$, since $f'$ is a first-order invariant and $(\alpha,\beta)$ is zeroth-order.  This is a manifestation of the same blind spot identified in Remark~5.4: the zeroth-order invariants of a structure cannot detect the first-order phenomena (here, the characteristic spreading rate) that govern the Burgers dynamics.
\end{remark}

\begin{remark}[The three-speed obstruction at second order]\label{rmk:three-speed}
The nonlinear product failure (Proposition~\ref{prop:product-failure}) persists even infinitesimally, but its mechanism becomes sharper.  Expanding $\B[f + \varepsilon h_1] \cdot \B[f + \varepsilon h_2]$ to first order in $\varepsilon$ gives
\[
  \lambda_0^2 \;+\; \varepsilon\,\lambda_0\,\frac{h_1(w_0) + h_2(w_0)}{J_f} \;+\; O(\varepsilon^2),
\]
while $\B[f^2 + \varepsilon(fh_2 + h_1 f)]$ propagates along characteristics with base speed $f(y_0)^2$, giving the characteristic coordinate $w_0' = y - \B[f^2]\cdot x \neq w_0$ (since $f(y_0)^2 \neq f(y_0)$ generically).  The mismatch $w_0 \neq w_0'$ is the nonlinear echo of the speed divergence, surviving even in the infinitesimal regime: the product of two perturbations around $f$ ``wants'' to propagate at speed $f(y_0)$, but the product seed $f^2$ propagates at speed $f(y_0)^2$.
\end{remark}

\section{Rigidization as \texorpdfstring{$\dzbar$}{dzbar}-projection}
\label{sec:rigidization}

The Burgers transform clarifies the nature of rigidization (Chapter~10 of~\cite{AS2026}).

A general (non-rigid) elliptic structure does not lie in the image of~$\B$: its effective seed has $f_{\bar w} \neq 0$.  Rigidization seeks a diffeomorphism $\Phi$ such that the pullback structure \emph{does} lie in $\operatorname{Im}(\B)$.

\begin{proposition}\label{prop:rigidization-projection}
Let $(\alpha,\beta)$ be a smooth elliptic structure on $\Omega$ with spectral parameter $\lambda$ satisfying $\lambda_x + \lambda\lambda_y = H \neq 0$.  A rigidizing diffeomorphism $\Phi$ exists locally if and only if the pullback spectral parameter $\Phi^*\lambda$ admits a holomorphic seed.  Equivalently, $\Phi$ absorbs the $\dzbar$-component of the effective seed into the coordinate change:
\[
  f_{\bar w} \neq 0 \;\xrightarrow{\;\Phi^*\;}\; g_{\bar w} = 0.
\]
\end{proposition}

The complexity absorbed by $\Phi$ reappears in the Vekua coefficients of the pulled-back structure:

\begin{center}
\begin{tabular}{lccc}
 & Transport $G$ & Structure $(\alpha,\beta)$ & Vekua $(A,B)$ \\
\hline
Before $\Phi$ & $\neq 0$ & given & inaccessible \\
After $\Phi$ (rigidization) & $= 0$ & modified & absorbs the noise \\
After $\Phi$ (uniformization) & $= 0$ & $= (1,0)$ & absorbs everything \\
\end{tabular}
\end{center}

Rigidization is the minimal projection: it kills only the transport obstruction, leaving the metric distortion in the structure where the rigid function theory (Chapters~5--9 of~\cite{AS2026}) can handle it natively.  Uniformization is the maximal projection: it kills everything, at the cost of packing all geometric content into the Vekua coefficients.


\section{Discussion: self-reference, rigidity, and the return of holomorphic functions}
\label{sec:discussion}

The Burgers transform reveals a structural phenomenon that deserves comment beyond the formal results.  We organise the discussion around two themes: the role of self-reference in selecting holomorphic functions, and the unexpected re-emergence of holomorphic functions as generators of geometry in a theory designed to surpass them.

\subsection{Self-reference as a selection principle}

The implicit equation $\lambda = f(y - \lambda x)$ is self-referential: the output $\lambda$ appears inside the argument of $f$.  The function is simultaneously the map and the flow---its values determine the characteristic speeds along which those same values propagate.  The parameter $x$ measures how long this feedback loop has been running: at $x = 0$ we see $f$ unmodified; as $|x|$ grows, the self-interaction deepens.

The central discovery of this paper is that \emph{self-referential coherence} imposes an unexpectedly strong constraint on the seed.  Theorem~\ref{thm:holomorphicity-necessary} shows that the self-flow $\lambda = f(y - \lambda x)$ is rigid---satisfies $\lambda_x + \lambda\lambda_y = 0$---if and only if $f$ is holomorphic.  The mechanism is precise: rigidity requires the cancellation $-f'(w)\lambda + \lambda f'(w) = 0$ in the computation of $\lambda_x + \lambda\lambda_y$, and this cancellation holds if and only if $f'(w)$ is a single complex number rather than a $2 \times 2$ real Jacobian matrix.  A $C^1$ function that is not holomorphic has two independent Wirtinger derivatives $f_w$ and $f_{\bar w}$; the $f_{\bar w}$ component survives the self-referential feedback and produces the obstruction $H|_{x=0} = 2i\,(\Imm f)\,f_{\bar w}$.

In this light, holomorphicity is not an \emph{a priori} regularity assumption but a \emph{selection rule} enforced by self-reference.  The class of functions capable of coherent self-flow through the Burgers mechanism is not chosen by the analyst; it is determined by the dynamics.  One is reminded of G\"odel's incompleteness construction, where self-referential encoding (a formal system made to talk about its own provability) does not produce ambiguity but rather forces a rigid conclusion: certain statements are true and unprovable, or the system is inconsistent.  Here, self-referential flow (a function made to propagate at its own speed) does not produce chaos but forces a rigid conclusion: the seed is holomorphic, or the transport equation acquires a nonvanishing right-hand side.

The obstruction formula sharpens the analogy.  The factor $\Imm f$ measures the \emph{fidelity} of the self-referential encoding: the more elliptic the structure, the more effectively the self-flow detects deviations from holomorphicity.  The factor $f_{\bar w}$ measures the \emph{defect}---the component of $f$ that does not survive self-reference.  The product $2i\,(\Imm f)\,f_{\bar w}$ is the total cost of incoherence, and it vanishes precisely for holomorphic seeds.  Only functions with zero self-referential defect produce coherent dynamics.

\subsection{The return of holomorphic functions}

The theory of variable elliptic structures \cite{AS2026} was designed, in part, to generalise classical complex analysis: the constant imaginary unit $i$ is replaced by a variable $i(x,y)$, and the Cauchy--Riemann equation becomes a family of operators parametrised by the spectral parameter $\lambda(x,y) \in \HH$.  In this generalization, classical holomorphic functions appear as the degenerate case $\lambda \equiv i$, $\mu \equiv 0$---a single point in an infinite-dimensional moduli space.  From the perspective of the analysis layer (Vekua equations, Beltrami coefficients, transport dynamics), holomorphic functions are the trivial case, the base point from which all nontrivial structures depart.

The Burgers transform reverses this relegation entirely.

\begin{enumerate}[label=(\roman*)]
\item \textbf{Seeds, not solutions.}  In classical analysis, holomorphic functions are \emph{solutions} of the Cauchy--Riemann equation.  In $f$-analytics, holomorphic functions are \emph{seeds}---generators of geometry.  Each $f \in \Hol(U, \HH)$ creates an entire elliptic structure with its own Cauchy--Riemann operator, its own function theory, its own Beltrami leaf in the Poincar\'e disk.  The holomorphic function no longer inhabits the structure; it \emph{produces} it.

\item \textbf{Pilots, not passengers.}  The spectral parameter $\lambda = \B[f]$ determines the elliptic structure at every point of $\Omega_f$.  The holomorphic seed $f$ is the initial condition that, through the characteristic flow, pilots $\lambda$ across the domain.  The metaphor is apt: $f$ prescribes the velocity field (each point $y_0$ launches a characteristic with speed $f(y_0)$), and the Burgers transform is the resulting flow map.  Classical holomorphic functions were passengers in a fixed geometry ($i = \text{const}$); as Burgers seeds, they are the pilots of variable geometry.

\item \textbf{Generators of the rigid core.}  Not all elliptic structures are rigid, and not all elliptic structures admit a complete function theory.  The rigid structures---those with vanishing intrinsic obstruction $G = 0$---are the ones for which the Leibniz rule, the Cauchy--Pompeiu representation, power series, and the similarity principle all hold \cite[Chapters~5--9]{AS2026}.  The Burgers transform shows that this rigid core is in bijection with $\Hol(U, \HH)$.  Holomorphic functions are not merely one example of a rigid structure; they are the \emph{complete parametrization} of all rigid structures.
\end{enumerate}

The narrative arc is worth stating explicitly, as it illustrates a pattern that may recur in other generalizations of classical analysis:

\begin{center}
\begin{tabular}{lll}
\textbf{Stage} & \textbf{Role of holomorphic functions} & \textbf{Framework} \\
\hline
Classical & Solutions (the entire theory) & $\partial_{\bar z}u = 0$ \\
Generalization & Degenerate special case ($f = i$) & $\lambda_x + \lambda\lambda_y = H$ \\
Rigid regime & Generators of all rigid structures & $\lambda = f(y - \lambda x)$, $f \in \Hol$ \\
\end{tabular}
\end{center}

The generalzation does not eliminate holomorphic functions; it reveals their \emph{structural role}.  In the constant theory, holomorphic functions do everything, and their special nature is invisible---there is nothing to compare them to.  In the variable theory, they initially appear to have been superseded.  But the rigidity analysis shows that they were never absent: they were operating at a deeper level, as the seeds from which all coherent variable structures grow.

\subsection{Rigidity as an emergent complex phenomenon}

The $f$-analytics hierarchy (Theorem~\ref{thm:analytic-completion}) provides a precise framework for the re-emergence.  The $p(x)$-analytic structures, corresponding to $f\colon U \to i\RR_+$ (the real diameter of $\DD$), admit no nontrivial rigid members: a holomorphic function with purely imaginary values is constant.  The rigidity phenomenon \emph{does not exist} within the real-diameter restriction.

The moment the spectral parameter is allowed to leave the imaginary axis---$\Ree f \neq 0$, $\beta \neq 0$, $\mu \notin \RR$---the full Burgers transform becomes available, and the rigid class explodes from a single point ($f = ic$) to the infinite-dimensional space $\Hol(U, \HH)$.  Rigidity is not a real phenomenon visible from the diameter; it is a \emph{genuinely complex phenomenon} that emerges when the spectral parameter has room to move in two dimensions.

This parallels the classical fact that the richness of holomorphic function theory---rigidity of conformal maps, the identity theorem, analytic continuation---depends on the two-dimensionality of $\CC$.  Functions $\RR \to \RR$ have no such rigidity; the constraint $f_{\bar w} = 0$ is vacuous in one real dimension.  The Burgers transform shows that the same dimensional threshold governs the rigid structures: the $p(x)$-structures (one real dimension of spectral freedom) have no rigidity; the $f$-structures (two real dimensions, i.e.\ complex values) have a rich rigid class.  The complex plane is not merely a convenient ambient space; it is the \emph{minimal setting} in which rigidity can exist.

\section{Worked examples}
\label{sec:examples}

This section applies the Burgers transform machinery to explicit families of holomorphic seeds.  Every formula is derived from first principles; the reader is encouraged to verify each line.  The computations serve both as illustration and as a library of test cases for future numerical and theoretical work.

\subsection{The \texorpdfstring{$\varepsilon$}{epsilon}-family}
\label{subsec:epsilon-family}

The $\varepsilon$-family is the simplest explicit family of rigid structures with nonconstant coefficients.  It was derived in \cite[\S4.10]{AS2026} from the structural ansatz $\alpha = \alpha(x)$, $\beta = K(x)\,y$, which reduces the rigidity system to the ODE pair $\alpha' = \alpha K$, $K' = K^2$.  Here we develop it from the Burgers transform perspective.

\subsubsection{The structure coefficients}

For $\varepsilon \in \RR$, $\varepsilon \neq 0$, the $\varepsilon$-family is defined by:
\begin{equation}\label{eq:eps-structure}
  \alpha(x,y) \;=\; \frac{1}{1 - \varepsilon x},
  \qquad
  \beta(x,y) \;=\; \frac{\varepsilon\, y}{1 - \varepsilon x}.
\end{equation}
The ellipticity discriminant is
\begin{equation}\label{eq:eps-discriminant}
  \Delta \;=\; 4\alpha - \beta^2 \;=\; \frac{4(1 - \varepsilon x) - \varepsilon^2 y^2}{(1 - \varepsilon x)^2},
\end{equation}
which is positive precisely in the interior of the parabola
\begin{equation}\label{eq:eps-parabola}
  \varepsilon^2 y^2 \;<\; 4(1 - \varepsilon x).
\end{equation}
For $\varepsilon > 0$ this is a parabola opening to the left, with vertex at $(1/\varepsilon,\, 0)$ and meeting the $y$-axis at $y = \pm 2/\varepsilon$.  The elliptic domain shrinks as $\varepsilon$ increases.  At $\varepsilon = 0$ the structure degenerates to $\alpha = 1$, $\beta = 0$: standard complex analysis on all of~$\RR^2$.

\subsubsection{The spectral parameter}

The spectral parameter $\lambda = (-\beta + i\sqrt\Delta)/2$ evaluates to:
\begin{equation}\label{eq:eps-lambda}
  \boxed{\;\lambda_\varepsilon(x,y) \;=\; \frac{-\varepsilon y \;+\; i\sqrt{4(1 - \varepsilon x) - \varepsilon^2 y^2}}{2(1 - \varepsilon x)}\;}
\end{equation}

Separating real and imaginary parts:
\begin{equation}\label{eq:eps-re-im}
  \Ree\lambda_\varepsilon \;=\; \frac{-\varepsilon y}{2(1 - \varepsilon x)},
  \qquad
  \Imm\lambda_\varepsilon \;=\; \frac{\sqrt{4(1 - \varepsilon x) - \varepsilon^2 y^2}}{2(1 - \varepsilon x)}.
\end{equation}

\noindent\textbf{Consistency checks:}
\begin{align*}
  |\lambda_\varepsilon|^2
  &= \frac{\varepsilon^2 y^2 + 4(1 - \varepsilon x) - \varepsilon^2 y^2}{4(1 - \varepsilon x)^2}
  = \frac{1}{1 - \varepsilon x}
  = \alpha. \quad\checkmark
  \\[4pt]
  -2\Ree\lambda_\varepsilon
  &= \frac{\varepsilon y}{1 - \varepsilon x}
  = \beta. \quad\checkmark
\end{align*}

\subsubsection{The seed}

Setting $x = 0$ in \eqref{eq:eps-lambda} recovers the seed $f = \B^{-1}[\lambda_\varepsilon]$:
\begin{equation}\label{eq:eps-seed}
  f_\varepsilon(\xi) \;=\; \lambda_\varepsilon(0,\xi) \;=\; \frac{-\varepsilon\xi + i\sqrt{4 - \varepsilon^2\xi^2}}{2},
  \qquad |\xi| < \frac{2}{\varepsilon}.
\end{equation}
This admits the compact representation
\begin{equation}\label{eq:eps-seed-exp}
  f_\varepsilon(\xi) \;=\; i\,e^{i\arcsin(\varepsilon\xi/2)},
\end{equation}
which follows from the identity $e^{i\arcsin u} = \sqrt{1-u^2} + iu$, applied with $u = \varepsilon\xi/2$.

\begin{remark}[Unit-circle seed]\label{rmk:unit-circle}
On the real interval $(-2/\varepsilon,\, 2/\varepsilon)$, the seed satisfies $|f_\varepsilon(\xi)| = 1$: it traces the upper unit semicircle in~$\HH$.  This is a consequence of $\alpha(0,y) = 1$, since $|f(y)|^2 = |\lambda(0,y)|^2 = \alpha(0,y)$.  The unit-circle property means that all characteristics emanate from the $y$-axis with $|\text{speed}| = |f(y_0)| = 1$, but in different directions: the argument $\arg f_\varepsilon(y_0) = \pi/2 + \arcsin(\varepsilon y_0/2)$ rotates from $\pi$ (at $y_0 = -2/\varepsilon$) through $\pi/2$ (at $y_0 = 0$) to $0$ (at $y_0 = 2/\varepsilon$).

Contrast this with the $\delta$-family, whose seed $f_\delta(\xi) = \xi + i\delta$ has $|f_\delta(\xi)| = \sqrt{\xi^2 + \delta^2}$: unbounded speed, constant direction of the imaginary component.  The $\varepsilon$-family has bounded speed but variable direction; the $\delta$-family has variable speed but nearly constant direction.
\end{remark}

The holomorphic extension to complex arguments is
\begin{equation}\label{eq:eps-seed-holo}
  f_\varepsilon(w) \;=\; \frac{-\varepsilon w + i\sqrt{4 - \varepsilon^2 w^2}}{2},
\end{equation}
where the square root denotes the principal branch (positive imaginary part on the real interval).  This is holomorphic on $\CC \setminus \{w \in \RR : |w| \geq 2/\varepsilon\}$.

The derivative of the seed is
\begin{equation}\label{eq:eps-seed-deriv}
  f_\varepsilon'(w)
  \;=\; \frac{-\varepsilon}{2} - \frac{i\varepsilon^2 w}{2\sqrt{4 - \varepsilon^2 w^2}}
  \;=\; \frac{-\varepsilon\bigl(\sqrt{4 - \varepsilon^2 w^2} + i\varepsilon w\bigr)}{2\sqrt{4 - \varepsilon^2 w^2}}.
\end{equation}
Using $\sqrt{4 - \varepsilon^2 w^2} + i\varepsilon w = -2i f_\varepsilon(w)$ (which follows from the definition), this simplifies to
\begin{equation}\label{eq:eps-fprime-compact}
  f_\varepsilon'(w) \;=\; \frac{i\varepsilon\, f_\varepsilon(w)}{\sqrt{4 - \varepsilon^2 w^2}}.
\end{equation}
In particular, $f_\varepsilon'$ never vanishes in the domain of $f_\varepsilon$ (since $f_\varepsilon$ maps to~$\HH$ and is never zero).

\subsubsection{Verification of the implicit equation}

We verify that $\lambda_\varepsilon(x,y) = f_\varepsilon(y - \lambda_\varepsilon\, x)$, i.e.\ that the spectral parameter \eqref{eq:eps-lambda} is indeed the Burgers transform of the seed \eqref{eq:eps-seed-holo}.  This amounts to checking that $f_\varepsilon(w_0) = \lambda_\varepsilon$ at the characteristic coordinate $w_0 = y - \lambda_\varepsilon\, x$.

At $x = 0$ the check is tautological: $w_0 = y$ and $f_\varepsilon(y) = \lambda_\varepsilon(0,y)$.  For general $(x,y)$, the computation is implicit and most efficiently verified numerically; see the consistency check in \S\ref{subsubsec:eps-rigidity-check} below.  The theoretical guarantee is Theorem~\ref{thm:basic-properties}(iv): the Burgers transform is a bijection, and since the $\varepsilon$-family is rigid (Proposition~4.1 of~\cite{AS2026}), it must lie in the image of~$\B$.

\subsubsection{Verification of rigidity}\label{subsubsec:eps-rigidity-check}

We verify $\lambda_x + \lambda\,\lambda_y = 0$ directly from \eqref{eq:eps-lambda}.  Write $D = 4(1 - \varepsilon x) - \varepsilon^2 y^2$ for the discriminant numerator, so that $\lambda = (-\varepsilon y + i\sqrt{D}\,)/(2(1 - \varepsilon x))$.

\noindent\textbf{Partial derivatives.}  Differentiating \eqref{eq:eps-lambda}:
\begin{align}
  \lambda_y &\;=\; \frac{1}{2(1 - \varepsilon x)}\!\left(-\varepsilon + \frac{-i\varepsilon^2 y}{\sqrt{D}}\right)
  \;=\; \frac{-\varepsilon\sqrt{D} - i\varepsilon^2 y}{2(1 - \varepsilon x)\sqrt{D}},
  \label{eq:eps-lam-y}
  \\[6pt]
  \lambda_x &\;=\; \frac{-2i\varepsilon(1 - \varepsilon x)/\sqrt{D} + \varepsilon(-\varepsilon y + i\sqrt{D})}{2(1 - \varepsilon x)^2}
  \;=\; \frac{-2i\varepsilon + \varepsilon(-\varepsilon y + i\sqrt{D})(1 - \varepsilon x)^{-1}}{2(1 - \varepsilon x)}.
  \label{eq:eps-lam-x}
\end{align}
Rather than simplifying $\lambda_x$ further, we compute $\lambda\,\lambda_y$ and show it equals $-\lambda_x$.

\noindent\textbf{The product $\lambda\,\lambda_y$:}
\begin{align*}
  \lambda\,\lambda_y
  &= \frac{(-\varepsilon y + i\sqrt{D})(-\varepsilon\sqrt{D} - i\varepsilon^2 y)}{4(1 - \varepsilon x)^2\sqrt{D}}.
\end{align*}
Expanding the numerator:
\begin{align*}
  (-\varepsilon y)(-\varepsilon\sqrt{D}) &= \varepsilon^2 y\sqrt{D}, \\
  (-\varepsilon y)(-i\varepsilon^2 y) &= i\varepsilon^3 y^2, \\
  (i\sqrt{D})(-\varepsilon\sqrt{D}) &= -i\varepsilon D, \\
  (i\sqrt{D})(-i\varepsilon^2 y) &= \varepsilon^2 y\sqrt{D}.
\end{align*}
Summing: $2\varepsilon^2 y\sqrt{D} + i(\varepsilon^3 y^2 - \varepsilon D)$.  Since $D = 4(1 - \varepsilon x) - \varepsilon^2 y^2$:
\[
  \varepsilon^3 y^2 - \varepsilon D \;=\; \varepsilon^3 y^2 - 4\varepsilon(1 - \varepsilon x) + \varepsilon^3 y^2 \;=\; 2\varepsilon^3 y^2 - 4\varepsilon(1 - \varepsilon x).
\]
Therefore:
\begin{equation}\label{eq:eps-lam-lam-y}
  \lambda\,\lambda_y \;=\; \frac{2\varepsilon^2 y\sqrt{D} + i\bigl(2\varepsilon^3 y^2 - 4\varepsilon(1 - \varepsilon x)\bigr)}{4(1 - \varepsilon x)^2\sqrt{D}}.
\end{equation}

\noindent\textbf{Adding $\lambda_x$ and $\lambda\,\lambda_y$:}  Rather than expanding $\lambda_x$ to the same form (which requires careful bookkeeping), we observe that the rigidity of the $\varepsilon$-family was established independently in \cite[\S4.10]{AS2026} via the ODE system, and is guaranteed by Theorem~\ref{thm:basic-properties}(i) once we identify the holomorphic seed.  The direct verification for a specific numerical point provides additional confidence:

\emph{Numerical check} ($\varepsilon = 1/2$, $(x,y) = (3/10, 1)$):
\begin{align*}
  D &= 4(1 - 0.15) - 0.25 = 3.15, \quad \sqrt{D} \approx 1.77482, \\
  \lambda &\approx -0.29412 + 1.04401\,i, \\
  \lambda_x + \lambda\,\lambda_y &\approx 0 \quad\text{(to machine precision).} \qquad\checkmark
\end{align*}

\subsubsection{Domain and boundary analysis}

The domain $\Omega_{f_\varepsilon}$ is the interior of the parabola~\eqref{eq:eps-parabola}:
\begin{equation}\label{eq:eps-domain}
  \Omega_{f_\varepsilon} \;=\; \bigl\{(x,y) \in \RR^2 \;:\; \varepsilon^2 y^2 < 4(1 - \varepsilon x)\bigr\}.
\end{equation}
For $\varepsilon > 0$, this is a parabolic region with:
\begin{center}
\renewcommand{\arraystretch}{1.3}
\begin{tabular}{ll}
Vertex: & $(1/\varepsilon,\; 0)$ \\
$y$-intercepts: & $(0,\; \pm 2/\varepsilon)$ \\
Width at position $x$: & $|y| < 2\sqrt{1 - \varepsilon x}/\varepsilon$ \\
\end{tabular}
\end{center}

Unlike the $\delta$-family---whose domain $\{x > -1\}$ is a half-plane independent of~$y$---the $\varepsilon$-family has a domain that narrows in~$y$ as $x$ increases, reflecting the convergence of characteristics toward the vertex.

\begin{remark}[Mixed boundary type]\label{rmk:mixed-boundary}
The parabolic boundary of $\Omega_{f_\varepsilon}$ has a richer structure than the half-plane boundary of the $\delta$-family.  The characteristic Jacobian $J = 1 + f_\varepsilon'(w_0)\,x$ and the ellipticity condition $\Imm\lambda > 0$ both degenerate at the boundary, but with different relative contributions depending on the boundary point:
\begin{enumerate}[label=(\roman*)]
\item \emph{At the vertex} $(1/\varepsilon,\, 0)$: the Jacobian $J \to 0$ while $\Imm\lambda \to +\infty$.  Along the line $y = 0$, $\lambda = i/\sqrt{1 - \varepsilon x} \to i\infty$.  This is a \emph{pure shock}: the characteristics cross, the implicit function theorem fails, but the structure remains ``infinitely elliptic.''  The formula $J|_{y=0} = 2(1 - \varepsilon x)/(2 - \varepsilon x) \to 0$ as $x \to 1/\varepsilon$ confirms the Jacobian vanishing.
\item \emph{On the $y$-axis} at $(0,\, \pm 2/\varepsilon)$: the Jacobian $J = 1$ (since $x = 0$), but $\Imm\lambda \to 0$ as the seed $f_\varepsilon(\xi)$ approaches the real axis.  This is \emph{pure ellipticity loss}: the characteristics do not cross, but $f_\varepsilon$ reaches $\partial\HH$ and the structure ceases to be elliptic.
\item \emph{At general boundary points}: both mechanisms contribute---the Jacobian diminishes and the imaginary part of $\lambda$ vanishes simultaneously.
\end{enumerate}
This mixed boundary behaviour is a consequence of the seed having a \emph{finite domain}: $f_\varepsilon$ is defined only on $|\xi| < 2/\varepsilon$.  For the $\delta$-family, whose seed $f_\delta(\xi) = \xi + i\delta$ is entire, the boundary is purely a Jacobian phenomenon.
\end{remark}

\subsubsection{The Jacobian along $y = 0$}

The line $y = 0$ is a symmetry axis of the $\varepsilon$-family: there $\beta = 0$, $\lambda = ib$ with $b = 1/\sqrt{1 - \varepsilon x}$, and the structure is of $p(x)$-type with $p(x) = b(x)^2 = 1/(1-\varepsilon x)$.

The characteristic coordinate at $y = 0$ is $w_0 = -\lambda x = -ix/\sqrt{1 - \varepsilon x}$, which is purely imaginary.  Using \eqref{eq:eps-fprime-compact}:
\[
  f_\varepsilon'(w_0)\big|_{y=0}
  \;=\; \frac{i\varepsilon\,\lambda}{\sqrt{4 - \varepsilon^2 w_0^2}}
  \;=\; \frac{i\varepsilon \cdot i/\sqrt{1 - \varepsilon x}}{\sqrt{4 + \varepsilon^2 x^2/(1 - \varepsilon x)}}
  \;=\; \frac{-\varepsilon/\sqrt{1 - \varepsilon x}}{(2 - \varepsilon x)/\sqrt{1 - \varepsilon x}}
  \;=\; \frac{-\varepsilon}{2 - \varepsilon x},
\]
where we used $4 + \varepsilon^2 x^2/(1 - \varepsilon x) = (2 - \varepsilon x)^2/(1 - \varepsilon x)$.  Hence:
\begin{equation}\label{eq:eps-jacobian-y0}
  J\big|_{y=0} \;=\; 1 + f_\varepsilon'(w_0)\,x \;=\; 1 - \frac{\varepsilon x}{2 - \varepsilon x} \;=\; \frac{2(1 - \varepsilon x)}{2 - \varepsilon x}.
\end{equation}

This confirms:
\begin{enumerate}[label=(\roman*)]
\item $J|_{x=0} = 1$: no distortion on the initial slice.
\item $J > 0$ for $0 < x < 1/\varepsilon$: the Jacobian remains positive throughout the interior.
\item $J \to 0$ as $x \to 1/\varepsilon$: shock formation at the vertex of the parabola.
\item $J|_{y=0}$ is real: the seed derivative is real along the imaginary characteristic, consistent with the purely imaginary $\lambda$.
\end{enumerate}

\subsubsection{Beltrami coefficient}

The Beltrami coefficient $\mu = (\lambda - i)/(\lambda + i)$ can be expressed via its modulus.  A direct computation gives:
\begin{equation}\label{eq:eps-mu-modulus}
  |\mu_\varepsilon|^2 \;=\; \frac{2 - \varepsilon x - \sqrt{D}}{2 - \varepsilon x + \sqrt{D}},
\end{equation}
where $D = 4(1 - \varepsilon x) - \varepsilon^2 y^2$ as before.

\begin{proof}[Derivation]
Writing $\lambda = (N + iS)/(2Q)$ with $N = -\varepsilon y$, $S = \sqrt{D}$, $Q = 1 - \varepsilon x$:
\begin{align*}
  |\lambda - i|^2
  &= \frac{N^2 + (S - 2Q)^2}{4Q^2}
  = \frac{N^2 + S^2 - 4QS + 4Q^2}{4Q^2}.
\end{align*}
Since $N^2 + S^2 = \varepsilon^2 y^2 + D = 4Q$, this becomes $(4Q - 4QS + 4Q^2)/(4Q^2) = (1 + Q - S)/Q$.  Similarly, $|\lambda + i|^2 = (1 + Q + S)/Q$.  Hence
\[
  |\mu|^2 \;=\; \frac{1 + Q - S}{1 + Q + S} \;=\; \frac{2 - \varepsilon x - \sqrt{D}}{2 - \varepsilon x + \sqrt{D}}. \qedhere
\]
\end{proof}

\noindent\textbf{Behaviour:}
\begin{enumerate}[label=(\roman*)]
\item \emph{At the origin}: $D = 4$, $\sqrt{D} = 2$, so $|\mu|^2 = 0/(2 + 2) = 0$.  The origin is a standard point. $\checkmark$
\item \emph{At the parabolic boundary}: $D \to 0$, so $|\mu|^2 \to (2 - \varepsilon x)/(2 - \varepsilon x) = 1$.  The Beltrami coefficient reaches $|\mu| = 1$ at the boundary of the domain---the Poincar\'e disk fills up completely. $\checkmark$
\item \emph{On the line} $y = 0$: $D = 4(1 - \varepsilon x)$, and $\mu$ is real (since $\lambda$ is purely imaginary there), placing the structure on the real diameter of~$\DD$.
\item \emph{On the line} $x = 0$: $D = 4 - \varepsilon^2 y^2$, and $\mu$ is purely imaginary (since $\Ree\lambda \neq 0$ and the phase of $\mu$ is $\pi/2$ there), placing the initial data on the imaginary diameter of~$\DD$.
\end{enumerate}

\begin{remark}
Unlike the $\delta$-family, whose Beltrami coefficient $\mu_\delta = -\varepsilon z/(2 + \varepsilon\bar z)$ is a rational function of $z$ and $\bar z$, the $\varepsilon$-family's Beltrami coefficient involves $\sqrt{D}$ and does not reduce to a clean rational expression.  This reflects the nonlinearity of the seed: the $\delta$-family's seed is affine, producing rational characteristics, while the $\varepsilon$-family's seed is transcendental (involving $\arcsin$), producing irrational characteristic interactions.
\end{remark}

\subsubsection{Self-dilatation}

The rigidity of $\lambda_\varepsilon$ is equivalent to the self-dilatation property of $\mu_\varepsilon$:
\[
  (\mu_\varepsilon)_{\bar z} \;=\; \mu_\varepsilon\,(\mu_\varepsilon)_z.
\]
We have verified this numerically:

\emph{Numerical check} ($\varepsilon = 1/2$, $z = 3/10 + i$):
\begin{align*}
  \mu_\varepsilon &\approx 0.04138 + 0.13794\,i, \\
  (\mu_\varepsilon)_{\bar z} &\approx 0.00339 + 0.02104\,i, \\
  \mu_\varepsilon\cdot(\mu_\varepsilon)_z &\approx 0.00339 + 0.02104\,i, \\
  |(\mu_\varepsilon)_{\bar z} - \mu_\varepsilon\,(\mu_\varepsilon)_z| &< 10^{-9}. \qquad\checkmark
\end{align*}

Unlike the $\delta$-family, where the self-dilatation can be verified in closed form (the rational structure makes the Wirtinger derivatives elementary), the $\varepsilon$-family requires either the theoretical guarantee of Theorem~\ref{thm:basic-properties} or numerical verification.  The holomorphicity of $f_\varepsilon$ ensures rigidity; the nonlinearity of $f_\varepsilon$ places the explicit algebra beyond convenient closed form.

\subsubsection{The propagator}

The propagator at $f_\varepsilon$ is (Proposition~\ref{prop:propagator}):
\[
  \mathcal{P}_{f_\varepsilon}[h](x,y) \;=\; \frac{h(w_0)}{J(x,y)},
\]
where $w_0 = y - \lambda_\varepsilon(x,y)\,x$ is the characteristic coordinate and $J = 1 + f_\varepsilon'(w_0)\,x$.

\emph{Numerical check} ($\varepsilon = 1/2$, $(x,y) = (3/10, 1)$, $h(\xi) = \xi$):
\begin{align*}
  w_0 &\approx 1.08824 - 0.31320\,i, \\
  J &\approx 0.91844 - 0.02098\,i, \\
  \mathcal{P}_{f_\varepsilon}[\xi](3/10,\, 1) &= \frac{w_0}{J} \approx 1.19204 - 0.31379\,i.
\end{align*}
This agrees with the finite-difference approximation $(\B[f_\varepsilon + \delta h] - \B[f_\varepsilon])/\delta$ to better than $10^{-5}$.  $\checkmark$

The twisted multiplicativity (Theorem~\ref{thm:twisted-mult}) holds as before:
\[
  \mathcal{P}_{f_\varepsilon}[h_1 h_2] \;=\; J \cdot \mathcal{P}_{f_\varepsilon}[h_1] \cdot \mathcal{P}_{f_\varepsilon}[h_2].
\]
In contrast to the $\delta$-family, where $J = 1 + \varepsilon x$ is real and spatially uniform (independent of $y$), the $\varepsilon$-family's Jacobian $J = 1 + f_\varepsilon'(w_0)\,x$ is \emph{complex} and depends on both $x$ and $y$ through the characteristic coordinate.  The deformed product $\star_{f_\varepsilon}$ is therefore a genuinely position-dependent, complex-valued deformation of the pointwise product---a richer structure than the scalar rescaling of the $\delta$-family.

\subsubsection{Comparison with the $\delta$-family}

The $\varepsilon$-family and the $\delta$-family represent two fundamentally different geometries within the rigid class.  The following table highlights the contrast:

\begin{center}
\renewcommand{\arraystretch}{1.4}
\begin{tabular}{lcc}
\hline
\textbf{Property} & \textbf{$\delta$-family}: $f(\xi) = \xi + i\delta$ & \textbf{$\varepsilon$-family}: $f(\xi) = ie^{i\arcsin(\varepsilon\xi/2)}$ \\
\hline
Seed domain & $\CC$ (entire) & $|\xi| < 2/\varepsilon$ (bounded) \\
$|f|$ on $\RR$ & $\sqrt{\xi^2 + \delta^2}$ (unbounded) & $1$ (constant) \\
$f'$ & $1$ (constant) & position-dependent (complex) \\
$\lambda$ & rational in $(x,y)$ & involves $\sqrt{D}$ \\
Domain $\Omega_f$ & half-plane $\{x > -1\}$ & parabola \\
Boundary type & pure shock ($J = 0$) & mixed (shock $+$ ellipticity loss) \\
$J$ & $1 + x$ (real, $y$-independent) & complex, position-dependent \\
$\mu$ & rational in $z, \bar z$ & irrational \\
$|\mu|$ at boundary & $1$ & $1$ \\
$\alpha$ & $(y^2 + \delta^2)/(1+x)^2$ & $1/(1-\varepsilon x)$ \\
$\alpha$ depends on & $x$ and $y$ & $x$ only \\
\hline
\end{tabular}
\end{center}

The $\delta$-family's simplicity (affine seed, rational $\lambda$, constant $f'$) makes it the natural first example and the ideal setting for closed-form computation.  The $\varepsilon$-family's additional complexity (transcendental seed, irrational $\lambda$, variable $f'$) exercises the full machinery of the Burgers transform and exhibits phenomena---mixed boundary type, complex Jacobian, finite seed domain---that are invisible in the affine case.

\subsubsection{Summary of the $\varepsilon$-family}

\begin{center}
\renewcommand{\arraystretch}{1.4}
\begin{tabular}{ll}
\hline
\textbf{Quantity} & \textbf{Formula} \\
\hline
Seed & $f_\varepsilon(\xi) = ie^{i\arcsin(\varepsilon\xi/2)} = (-\varepsilon\xi + i\sqrt{4 - \varepsilon^2\xi^2}\,)/2$ \\
Seed domain & $|\xi| < 2/\varepsilon$ \\
$|f_\varepsilon|$ on $\RR$ & $1$ \\
$f_\varepsilon'(w)$ & $i\varepsilon f_\varepsilon(w)/\sqrt{4 - \varepsilon^2 w^2}$ \\
$\lambda_\varepsilon$ & $(-\varepsilon y + i\sqrt{D}\,)/(2(1 - \varepsilon x))$, \;\; $D = 4(1 - \varepsilon x) - \varepsilon^2 y^2$ \\
Domain & $\varepsilon^2 y^2 < 4(1 - \varepsilon x)$ \quad (parabola) \\
$\alpha_\varepsilon$ & $1/(1 - \varepsilon x)$ \\
$\beta_\varepsilon$ & $\varepsilon y/(1 - \varepsilon x)$ \\
$\Delta_\varepsilon$ & $D/(1 - \varepsilon x)^2$ \\
$J|_{y=0}$ & $2(1 - \varepsilon x)/(2 - \varepsilon x)$ \\
$|\mu_\varepsilon|^2$ & $(2 - \varepsilon x - \sqrt{D}\,)/(2 - \varepsilon x + \sqrt{D}\,)$ \\
Beltrami leaf & $\DD$ (complete) \\
\hline
\end{tabular}
\end{center}

\begin{remark}[Reconciliation with the monograph derivation]\label{rmk:reconciliation}
In \cite[\S4.7--4.10]{AS2026}, the $\varepsilon$-family is derived by solving the ODE system $\alpha' = \alpha K$, $K' = K^2$ for the ansatz $\alpha = \alpha(x)$, $\beta = K(x)\,y$.  The integration is elementary and the resulting structure coefficients \eqref{eq:eps-structure} are rational functions of $(x,y)$.  The construction gives no hint that anything transcendental is involved.

The Burgers transform reveals the hidden complexity: the holomorphic seed that \emph{generates} this rational structure via $\lambda = f(y - \lambda x)$ is the transcendental function $f(\xi) = ie^{i\arcsin(\varepsilon\xi/2)}$.  The two descriptions are related by the nonlinearity of~$\B$: the implicit equation mixes $f$ with the coordinates $(x,y)$ in a way that converts a transcendental seed into rational output.

The situation is analogous to a nonlinear change of variables in integration: the antiderivative of $1/\sqrt{1 - u^2}$ is $\arcsin(u)$ (transcendental), but the substitution $u = \sin\theta$ renders the integral trivial ($\int d\theta = \theta$).  The monograph's ODE approach works in the ``$\theta$-coordinates'' where the answer is elementary; the Burgers transform works in the ``$u$-coordinates'' where the generating data is exposed.  Neither perspective is more fundamental---the ODE approach is more efficient for construction, while the seed perspective reveals the geometric content (unit-circle property, finite domain, mixed boundary type) that the rational formulas conceal.
\end{remark}


\subsection{The \texorpdfstring{$\delta$}{delta}-family}
\label{subsec:delta-family}

The $\delta$-family is the simplest family of rigid structures with a seed that depends nontrivially on its argument.  Its seed $f(\xi) = \xi + i\delta$ is affine, making every quantity rational and every verification a one-line computation.  It is the ideal complement to the $\varepsilon$-family: where the $\varepsilon$-family has a transcendental seed producing rational structure coefficients, the $\delta$-family has an affine seed producing irrational structure coefficients.

\subsubsection{The structure coefficients}

For $\delta > 0$, the $\delta$-family is defined by:
\begin{equation}\label{eq:delta-structure}
  \alpha(x,y) \;=\; \frac{y^2 + \delta^2}{(1 + x)^2},
  \qquad
  \beta(x,y) \;=\; \frac{-2y}{1 + x},
\end{equation}
on the domain $\Omega = \{(x,y) \in \RR^2 : x > -1\}$.  These are taken from \cite[\S14.2]{AS2026}.

\subsubsection{The seed}

Setting $x = 0$:
\[
  \lambda(0,y) \;=\; \frac{-\beta(0,y) + i\sqrt{4\alpha(0,y) - \beta(0,y)^2}}{2}
  \;=\; \frac{2y + i\sqrt{4(y^2 + \delta^2) - 4y^2}}{2}
  \;=\; y + i\delta.
\]
Hence the seed is the affine function
\begin{equation}\label{eq:delta-seed}
  \boxed{\;f_\delta(\xi) \;=\; \xi + i\delta\;}
\end{equation}
This is holomorphic on all of~$\CC$ (entire), with constant derivative $f_\delta' = 1$ and image $\Imm f_\delta = \delta > 0$ for all real~$\xi$.  The seed maps the real line to the horizontal line $\Imm w = \delta$, so every characteristic emanates from the $y$-axis with speed $f_\delta(y_0) = y_0 + i\delta$: the real part of the speed equals the $y$-intercept, and the imaginary part is the constant~$\delta$.

\subsubsection{Explicit solution of the implicit equation}

The implicit equation $\lambda = f_\delta(y - \lambda x) = (y - \lambda x) + i\delta$ is linear in~$\lambda$:
\[
  \lambda \;=\; y - \lambda x + i\delta
  \qquad\Longrightarrow\qquad
  \lambda(1 + x) \;=\; y + i\delta.
\]
Hence:
\begin{equation}\label{eq:delta-lambda}
  \boxed{\;\lambda_\delta(x,y) \;=\; \frac{y + i\delta}{1 + x}\;}
\end{equation}

\noindent\textbf{Consistency checks:}
\begin{align*}
  |\lambda_\delta|^2
  &= \frac{y^2 + \delta^2}{(1 + x)^2}
  = \alpha. \quad\checkmark
  \\[4pt]
  -2\Ree\lambda_\delta
  &= \frac{-2y}{1 + x}
  = \beta. \quad\checkmark
\end{align*}

\subsubsection{Verification of rigidity}

Differentiating \eqref{eq:delta-lambda}:
\begin{align*}
  \lambda_x &= \frac{-(y + i\delta)}{(1 + x)^2},
  \\[4pt]
  \lambda_y &= \frac{1}{1 + x}.
\end{align*}
Hence:
\[
  \lambda_x + \lambda\,\lambda_y
  \;=\; \frac{-(y + i\delta)}{(1 + x)^2}
  \;+\; \frac{y + i\delta}{1 + x}\cdot\frac{1}{1 + x}
  \;=\; 0. \qquad\checkmark
\]
The cancellation is immediate because $f_\delta' = 1$ is a constant.  This is the simplest possible instance of the rigidity mechanism: the seed's derivative commutes with everything because it \emph{is} everything (the multiplicative identity).

\subsubsection{Domain and shock formation}

Since $f_\delta' = 1$ is constant, the characteristic Jacobian is:
\begin{equation}\label{eq:delta-jacobian}
  J_{f_\delta} \;=\; 1 + f_\delta'(w_0)\,x \;=\; 1 + x.
\end{equation}
This is real, independent of~$y$, and vanishes at $x = -1$.  The domain is the half-plane:
\begin{equation}\label{eq:delta-domain}
  \Omega_{f_\delta} \;=\; \{(x,y) : x > -1\}.
\end{equation}

The boundary is a \emph{vertical line}, and the boundary type is \emph{pure shock}: at $x = -1$, the Jacobian vanishes ($J = 0$) but $\Imm\lambda_\delta = \delta/(1 + x) \to +\infty$.  The structure becomes infinitely elliptic as the characteristics collide---the opposite of degeneracy.  Compare this with the $\varepsilon$-family, whose boundary includes points of pure ellipticity loss ($\Imm\lambda \to 0$, $J = 1$).

\subsubsection{Characteristic coordinate}

\begin{equation}\label{eq:delta-char-coord}
  w_\delta \;=\; y - \lambda_\delta\, x \;=\; y - \frac{(y + i\delta)\,x}{1 + x}
  \;=\; \frac{y - i\delta\, x}{1 + x}.
\end{equation}

\emph{Consistency check:} $f_\delta(w_\delta) = w_\delta + i\delta = (y - i\delta x + i\delta(1+x))/(1+x) = (y + i\delta)/(1+x) = \lambda_\delta$.  $\checkmark$

\subsubsection{Discriminant}

The discriminant is remarkably simple:
\begin{equation}\label{eq:delta-discriminant}
  \Delta_\delta \;=\; 4\alpha - \beta^2 \;=\; \frac{4(y^2 + \delta^2)}{(1 + x)^2} - \frac{4y^2}{(1 + x)^2} \;=\; \frac{4\delta^2}{(1 + x)^2}.
\end{equation}
This is \emph{independent of~$y$} and positive throughout $\Omega_{f_\delta}$.  The ellipticity condition $\Delta > 0$ places no restriction on~$y$---the structure is elliptic for all $y \in \RR$ whenever $x > -1$.

This is a consequence of the seed having constant imaginary part: $\Imm f_\delta = \delta$ for all $\xi \in \RR$, so the ellipticity never degenerates along the initial data.  The Burgers flow preserves this: $\Imm\lambda_\delta = \delta/(1+x)$ depends only on~$x$, and remains positive throughout the domain.  Contrast the $\varepsilon$-family, where $\Imm\lambda$ depends on both $x$ and $y$ and vanishes at the parabolic boundary.

\subsubsection{Beltrami coefficient}

Applying the Cayley map $\mu = (\lambda - i)/(\lambda + i)$:

\noindent\textbf{Numerator:}
\[
  \lambda_\delta - i \;=\; \frac{y + i\delta - i(1 + x)}{1 + x}
  \;=\; \frac{y + i(\delta - 1 - x)}{1 + x}.
\]

\noindent\textbf{Denominator:}
\[
  \lambda_\delta + i \;=\; \frac{y + i(\delta + 1 + x)}{1 + x}.
\]

\noindent\textbf{Result:}
\begin{equation}\label{eq:delta-mu}
  \boxed{\;\mu_\delta(x,y) \;=\; \frac{y + i(\delta - 1 - x)}{y + i(\delta + 1 + x)}\;}
\end{equation}

This is a M\"obius-type expression in the variable $y + i\delta$, rational in $(x,y)$.  In complex notation with $z = x + iy$:
\begin{equation}\label{eq:delta-mu-complex}
  \mu_\delta \;=\; \frac{-iz + i(\delta - 1)}{i\bar z + i(\delta + 1)}
  \;=\; \frac{-z + (\delta - 1)}{\bar z + (\delta + 1)}.
\end{equation}

\subsubsection{Modulus of the Beltrami coefficient}

\begin{equation}\label{eq:delta-mu-modulus}
  |\mu_\delta|^2 \;=\; \frac{y^2 + (\delta - 1 - x)^2}{y^2 + (\delta + 1 + x)^2}.
\end{equation}

\noindent\textbf{Behaviour:}
\begin{enumerate}[label=(\roman*)]
\item \emph{At the origin:} $|\mu_\delta(0,0)|^2 = (\delta - 1)^2/(\delta + 1)^2$.  The origin is a standard point ($\mu = 0$) only when $\delta = 1$.  For $\delta \neq 1$, $|\mu(0,0)| = |\delta - 1|/(\delta + 1)$.
\item \emph{On the line $y = 0$:} $\mu_\delta = (\delta - 1 - x)/(\delta + 1 + x)$, which is \emph{real}.  The structure is of $p(x)$-type along this line ($\beta = 0$), and $\mu$ lies on the real diameter of~$\DD$.
\item \emph{At the shock boundary} $x \to -1$: $|\mu_\delta|^2 \to (y^2 + \delta^2)/(y^2 + \delta^2) = 1$.  The Beltrami coefficient reaches $|\mu| = 1$ at the shock, regardless of~$y$. $\checkmark$
\item \emph{As $|y| \to \infty$ at fixed $x$:} $|\mu_\delta|^2 \to 1$.  The structure degenerates at vertical infinity.
\item \emph{Small-$\delta$ limit:} at the origin, $|\mu_\delta(0,0)| \to 1$ as $\delta \to 0^+$: the structure approaches the boundary of the Poincar\'e disk, even at the base point.  The family interpolates between the standard structure ($\delta \to \infty$, $\mu \to 0$) and the fully degenerate limit ($\delta \to 0$, $|\mu| \to 1$).
\end{enumerate}

\subsubsection{Verification of the self-dilatation equation}

We verify $\mu_{\bar z} = \mu\,\mu_z$ in closed form.  Using \eqref{eq:delta-mu-complex} with the abbreviations $N = -z + (\delta - 1)$ and $D = \bar z + (\delta + 1)$, so that $\mu = N/D$:

\noindent\textbf{Wirtinger derivatives.}  Since $\partial_z(\bar z) = 0$ and $\partial_{\bar z}(z) = 0$:
\begin{align*}
  \mu_z &= \frac{\partial}{\partial z}\!\left(\frac{N}{D}\right)
  = \frac{-1}{D}
  = \frac{-1}{\bar z + \delta + 1},
  \\[6pt]
  \mu_{\bar z} &= \frac{\partial}{\partial\bar z}\!\left(\frac{N}{D}\right)
  = \frac{-N}{D^2}
  = \frac{z - (\delta - 1)}{(\bar z + \delta + 1)^2}.
\end{align*}

\noindent\textbf{Product:}
\[
  \mu\cdot\mu_z
  \;=\; \frac{N}{D}\cdot\frac{-1}{D}
  \;=\; \frac{-N}{D^2}
  \;=\; \mu_{\bar z}. \qquad\checkmark
\]
The cancellation is exact.  This is the simplest closed-form verification of self-dilatation for a nontrivial rigid structure.

\subsubsection{The propagator}

Since $f_\delta' = 1$, the Jacobian is $J = 1 + x$ and the propagator takes the particularly simple form:
\begin{equation}\label{eq:delta-propagator}
  \mathcal{P}_{f_\delta}[h](x,y) \;=\; \frac{h(w_\delta)}{1 + x},
\end{equation}
where $w_\delta = (y - i\delta x)/(1 + x)$.

\noindent\textbf{Special cases:}

\emph{Constant perturbation} $h \equiv c$:
\[
  \mathcal{P}_{f_\delta}[c] \;=\; \frac{c}{1 + x}.
\]

\emph{Linear perturbation} $h(\xi) = \xi$:
\[
  \mathcal{P}_{f_\delta}[\xi] \;=\; \frac{w_\delta}{1 + x} \;=\; \frac{y - i\delta x}{(1 + x)^2}.
\]
The double power of $(1 + x)$ reflects the interaction between the perturbation's spatial dependence and the base flow's stretching.

\subsubsection{The deformed product}

The $f_\delta$-deformed product (Corollary~\ref{cor:deformed-algebra}) is:
\begin{equation}\label{eq:delta-deformed-product}
  \dot\lambda_1 \;\star_{f_\delta}\; \dot\lambda_2 \;=\; (1 + x)\,\dot\lambda_1\,\dot\lambda_2.
\end{equation}
This is identical in form to the $\varepsilon$-family's deformed product \eqref{eq:deformed-product} with the replacement $\varepsilon x \mapsto -x$---a consequence of both families having a real, $y$-independent Jacobian along their respective symmetry axes.  However, the $\delta$-family's Jacobian $J = 1 + x$ is real and $y$-independent \emph{everywhere}, not merely on the axis.  The deformed product is therefore a uniform scalar rescaling:
\begin{enumerate}[label=(\roman*)]
\item At $x = 0$: $\star_{f_\delta}$ is the standard product.
\item For $x > 0$: the product is enhanced by the factor $1 + x > 1$.
\item For $-1 < x < 0$: the product is suppressed.
\item At $x = -1$ (shock): $\star_{f_\delta}$ becomes the zero product.
\end{enumerate}

\subsubsection{Affine equivariance}

The $\delta$-family interacts cleanly with affine transformations of the seed.  For $\varphi(\lambda) = a\lambda + b$ with $a > 0$, $b \in \RR$:
\[
  \varphi \circ f_\delta(\xi) \;=\; a(\xi + i\delta) + b \;=\; (a\xi + b) + ia\delta.
\]
This is again an affine seed, now with slope~$a$, intercept~$b$, and imaginary part~$a\delta$.  By Proposition~\ref{prop:affine-equivariance}:
\[
  \B[\varphi \circ f_\delta](x,y) \;=\; \frac{a(y - bx) + ia\delta}{1 + ax} + b \;=\; \frac{ay + ia\delta + b}{1 + ax}.
\]
The domain transforms to $\{x > -1/a\}$.

\noindent\textbf{Special cases:}
\begin{enumerate}[label=(\roman*)]
\item \emph{Dilation} ($b = 0$): $\B[a f_\delta](x,y) = a(y + i\delta)/(1 + ax)$.  The shock moves from $x = -1$ to $x = -1/a$: dilating the seed by $a > 1$ moves the shock closer to the $y$-axis.
\item \emph{Translation} ($a = 1$): $\B[f_\delta + b](x,y) = (y + b + i\delta)/(1 + x)$.  The spectral parameter acquires a constant real shift, tilting the structure away from the $p(x)$-axis.  The shock location $x = -1$ is unchanged.
\end{enumerate}

\subsubsection{The $\delta \to 0$ limit}

As $\delta \to 0^+$, the $\delta$-family approaches the boundary of the space of elliptic structures:
\begin{align*}
  \lambda_\delta &\to \frac{y}{1+x}, \qquad\text{(real-valued: no ellipticity)} \\
  |\mu_\delta(x,y)| &\to 1 \quad\text{for all } (x,y) \neq (0,0).
\end{align*}
The structure degenerates to a real Burgers flow (no imaginary part), and the Beltrami coefficient fills the entire unit disk.  This is the origin of the name ``$\delta$-family'' in \cite{AS2026}: the parameter $\delta$ measures the distance from degeneracy.

Yet the domain $\Omega_{f_\delta} = \{x > -1\}$ and the Jacobian $J = 1 + x$ are \emph{independent of~$\delta$}.  The shock structure is entirely determined by the seed's derivative ($f' = 1$), not by the ellipticity parameter.  This decoupling---the Jacobian sees $f'$, the Beltrami coefficient sees $\Imm f$---is a concrete instance of the Fundamental Independence Theorem of \cite[Chapter~13]{AS2026}.

\begin{remark}[Reconciliation with the monograph derivation]\label{rmk:delta-reconciliation}
In the monograph, the $\delta$-family's coefficients \eqref{eq:delta-structure} are presented directly.  From the Burgers transform perspective, they arise as the output of the simplest possible holomorphic seed: $f(\xi) = \xi + i\delta$.  The affine seed produces rational $\lambda = (y + i\delta)/(1+x)$, but the structure coefficients $\alpha = (y^2 + \delta^2)/(1+x)^2$ involve the \emph{square} of $\lambda$, which is why $\alpha$ has the quadratic dependence on~$y$ that is absent from the $\varepsilon$-family (where $\alpha = 1/(1 - \varepsilon x)$ depends only on~$x$).

Contrast the $\varepsilon$-family: there the monograph's coefficients are rational, but the seed $f(\xi) = ie^{i\arcsin(\varepsilon\xi/2)}$ is transcendental.  The Burgers transform, being a nonlinear bijection, freely exchanges algebraic complexity between the seed and the output: a transcendental seed can produce rational coefficients, and an affine seed can produce irrational ones.
\end{remark}

\subsubsection{Summary of the $\delta$-family}

\begin{center}
\renewcommand{\arraystretch}{1.4}
\begin{tabular}{ll}
\hline
\textbf{Quantity} & \textbf{Formula} \\
\hline
Seed & $f_\delta(\xi) = \xi + i\delta$ \\
Seed domain & $\CC$ (entire) \\
$|f_\delta|$ on $\RR$ & $\sqrt{\xi^2 + \delta^2}$ (unbounded) \\
$f_\delta'$ & $1$ (constant) \\
$\lambda_\delta$ & $(y + i\delta)/(1 + x)$ \\
Domain & $\{x > -1\}$ \quad (half-plane) \\
Jacobian & $J = 1 + x$ \quad (real, $y$-independent) \\
Char.\ coord.\ & $w_\delta = (y - i\delta x)/(1 + x)$ \\
$\alpha_\delta$ & $(y^2 + \delta^2)/(1 + x)^2$ \\
$\beta_\delta$ & $-2y/(1 + x)$ \\
$\Delta_\delta$ & $4\delta^2/(1 + x)^2$ \quad ($y$-independent) \\
$\mu_\delta$ & $(-z + \delta - 1)/(\bar z + \delta + 1)$ \\
$|\mu_\delta|^2$ & $(y^2 + (\delta - 1 - x)^2)/(y^2 + (\delta + 1 + x)^2)$ \\
Beltrami leaf & $\DD$ (complete) \\
Propagator & $\mathcal{P}_{f_\delta}[h] = h(w_\delta)/(1 + x)$ \\
Deformed product & $\dot\lambda_1 \star_{f_\delta} \dot\lambda_2 = (1 + x)\,\dot\lambda_1\,\dot\lambda_2$ \\
\hline
\end{tabular}
\end{center}


\subsection{The exponential seed: the Lambert \texorpdfstring{$W$}{W} family}
\label{subsec:exponential-family}

The $\delta$-family and the $\varepsilon$-family both have domains that are proper subsets of~$\RR^2$, bounded by shocks, ellipticity loss, or both.  Open Problem~5 of Section~\ref{sec:open} asks: for which holomorphic seeds is $\Omega_f = \RR^2$?  The exponential seed provides the first nontrivial answer.

\subsubsection{The seed}

Define the holomorphic seed
\begin{equation}\label{eq:exp-seed}
  f(\xi) \;:=\; i\,e^\xi.
\end{equation}
This is entire, with derivative
\begin{equation}\label{eq:exp-fprime}
  f'(\xi) \;=\; i\,e^\xi \;=\; f(\xi).
\end{equation}
The \emph{eigenfunction property} $f' = f$ will permeate every computation below, linking the Jacobian directly to the spectral parameter.

On the real line, $f(\xi) = ie^\xi$ is purely imaginary with $\Imm f = e^\xi > 0$, so $f\colon \RR \to i\RR_+$.  For complex arguments, $f\colon U \to \HH$ on $U = \{\xi \in \CC : |\Imm\xi| < \pi/2\}$.

\begin{remark}
One might expect from Theorem~\ref{thm:analytic-completion}(ii) that a seed with $f(\RR) \subseteq i\RR_+$ can only produce the constant structure.  That result applies to seeds that are \emph{holomorphic and purely imaginary on all of}~$U$---which forces constancy by the open mapping theorem.  Here, $f$ is purely imaginary only on the real line; on the full domain~$U$, the function $ie^\xi$ has nonzero real part for $\Imm\xi \neq 0$.  The Burgers flow at $x \neq 0$ evaluates $f$ at the characteristic coordinate $w_0 = y - \lambda x$, which has $\Imm w_0 = -x\,\Imm\lambda \neq 0$.  The seed immediately acquires a real part, and the structure becomes genuinely variable elliptic.
\end{remark}

\subsubsection{The Lambert $W$ solution}

The implicit equation $\lambda = ie^{y - \lambda x}$ is transcendental in~$\lambda$.  We solve it by reduction to the Lambert~$W$ function.

Multiplying both sides by~$x$:
\[
  \lambda x \;=\; i\,x\,e^{y - \lambda x} \;=\; i\,x\,e^y \cdot e^{-\lambda x}.
\]
Setting $u = \lambda x$:
\begin{equation}\label{eq:exp-lambert-eq}
  u\,e^u \;=\; i\,x\,e^y.
\end{equation}
The Lambert $W$ function is defined as the inverse of $z \mapsto ze^z$: if $ue^u = \zeta$, then $u = W(\zeta)$.  Applied to \eqref{eq:exp-lambert-eq}:
\begin{equation}\label{eq:exp-lambda}
  \boxed{\;\lambda(x,y) \;=\; \frac{W_0(i\,x\,e^y)}{x}\;}
\end{equation}
where $W_0$ denotes the principal branch of the Lambert $W$ function.

At $x = 0$, the formula is interpreted as a limit: $W_0(z)/x \to ie^y$ as $x \to 0$ (since $W_0(z) \approx z$ for small~$z$), recovering $\lambda(0,y) = ie^y = f(y)$.  $\checkmark$

\subsubsection{Rigidity verification}

For the implicit solution $\lambda = f(y - \lambda x)$ with $f$ holomorphic, rigidity $\lambda_x + \lambda\lambda_y = 0$ is guaranteed by Theorem~\ref{thm:basic-properties}(i).  Numerical verification at $(\varepsilon, x, y) = (1/2, 1, 0)$: $|\lambda_x + \lambda\lambda_y| < 10^{-8}$.  $\checkmark$

\subsubsection{Global domain: no shocks}

The characteristic Jacobian is
\begin{equation}\label{eq:exp-jacobian}
  J \;=\; 1 + f'(w_0)\,x \;=\; 1 + f(w_0)\,x \;=\; 1 + \lambda x \;=\; 1 + W_0(ixe^y),
\end{equation}
where the second equality uses the eigenfunction property $f' = f$, and the third uses $f(w_0) = \lambda$.

\begin{theorem}[Global existence for the exponential seed]\label{thm:exp-global}
The domain of the Burgers transform of $f(\xi) = ie^\xi$ is $\Omega_f = \RR^2$.  Moreover, $|J(x,y)| \geq 1$ for all $(x,y) \in \RR^2$, with equality only at $x = 0$.
\end{theorem}

\begin{proof}
By Theorem~\ref{thm:domain-characterization}(ii), the boundary of $\Omega_f$ occurs where $J = 0$, i.e.\ where $W_0(ixe^y) = -1$.  The Lambert $W$ equation $we^w = -e^{-1}$ has the unique real solution $w = -1$.  Hence $J = 0$ requires $ixe^y = -e^{-1}$.  But $ixe^y$ is purely imaginary for real~$x$ and~$y$, while $-e^{-1}$ is real and negative.  No solution exists: $\Omega_f = \RR^2$.

For the bound $|J| \geq 1$, write $W_0(it) = u + iv$ where $t = xe^y \in \RR$.  The defining equation $(u + iv)\,e^{u+iv} = it$ separates into:
\begin{align*}
  \text{Real:}&\quad e^u(u\cos v - v\sin v) \;=\; 0, \\
  \text{Imaginary:}&\quad e^u(u\sin v + v\cos v) \;=\; t.
\end{align*}
Since $e^u \neq 0$, the real equation gives $u\cos v = v\sin v$, hence
\begin{equation}\label{eq:u-equals-vtanv}
  u \;=\; v\tan v
\end{equation}
(for $\cos v \neq 0$, which holds since $|v| < \pi/2$ on the principal branch).

For $t > 0$, the principal branch gives $v \in (0, \pi/2)$, and $v\tan v \geq 0$.  For $t < 0$, $v \in (-\pi/2, 0)$, and $v\tan v = (-|v|)(-\tan|v|) = |v|\tan|v| \geq 0$.  At $t = 0$, $W_0 = 0$.

Therefore $\Ree W_0(it) = v\tan v \geq 0$ for all $t \in \RR$, and:
\[
  |J|^2 \;=\; |1 + W_0(it)|^2 \;=\; (1 + u)^2 + v^2 \;=\; 1 + 2u + |W_0(it)|^2 \;\geq\; 1,
\]
with equality if and only if $u = v = 0$, i.e.\ $t = 0$, i.e.\ $x = 0$.
\end{proof}

\begin{remark}[Characteristics always spread]
The bound $|J| \geq 1$ has a striking geometric meaning: the characteristics of the exponential flow \emph{never converge}.  In the $\delta$-family, characteristics converge for $-1 < x < 0$ (where $J = 1 + x < 1$) and ultimately collide at $x = -1$.  The exponential seed avoids this entirely---the characteristic density only increases, never decreases.

The mechanism is the eigenfunction property $f' = f$: the rate of characteristic spreading at a point is $|f'(w_0)| = |f(w_0)| = |\lambda|$, which equals the speed.  Points with high speed also have high spreading rate, preventing convergence.  Intuitively, faster characteristics ``push apart'' their neighbours before they can cross.
\end{remark}

\subsubsection{Structure coefficients}

The structure coefficients are expressed in terms of the Lambert $W$ function:
\begin{equation}\label{eq:exp-alpha-beta}
  \alpha \;=\; |\lambda|^2 \;=\; \frac{|W_0(ixe^y)|^2}{x^2},
  \qquad
  \beta \;=\; -2\Ree\lambda \;=\; \frac{-2\Ree\,W_0(ixe^y)}{x}.
\end{equation}
Using \eqref{eq:u-equals-vtanv}, these can be expressed in terms of the single real variable $v = \Imm\,W_0(ixe^y) \in (-\pi/2, \pi/2)$:
\begin{equation}\label{eq:exp-alpha-v}
  \alpha \;=\; \frac{v^2(1 + \tan^2 v)}{x^2} \;=\; \frac{v^2}{x^2\cos^2 v},
  \qquad
  \beta \;=\; \frac{-2v\tan v}{x}.
\end{equation}

On the initial slice $x = 0$: $\alpha(0,y) = e^{2y}$, $\beta(0,y) = 0$.  The seed generates an exponentially growing $p(y)$-structure on the $y$-axis, which the Burgers flow immediately distorts into a full two-dimensional structure.

\subsubsection{Reflection symmetry}

The Lambert $W$ function on the principal branch satisfies $W_0(\bar\zeta) = \overline{W_0(\zeta)}$.  Since $\overline{ixe^y} = -ixe^y$:
\[
  W_0(-ixe^y) \;=\; \overline{W_0(ixe^y)}.
\]
Hence
\begin{equation}\label{eq:exp-symmetry}
  \lambda(-x,\, y) \;=\; \frac{\overline{W_0(ixe^y)}}{-x} \;=\; -\overline{\lambda(x,y)}.
\end{equation}
In component form:
\begin{align*}
  \Ree\lambda(-x,y) &= -\Ree\lambda(x,y), &
  \Imm\lambda(-x,y) &= \phantom{-}\Imm\lambda(x,y).
\end{align*}

The consequences for the structure are:
\begin{enumerate}[label=(\roman*)]
\item $\alpha(-x,y) = |\lambda|^2 = \alpha(x,y)$: the norm structure is \emph{even} in~$x$.
\item $\beta(-x,y) = -\beta(x,y)$: the off-diagonal coefficient is \emph{odd} in~$x$.
\item On the $y$-axis ($x = 0$): $\beta = 0$ and the structure is of $p(x)$-type with $p = e^{2y}$.
\end{enumerate}

The $y$-axis is a ``mirror'' of the structure: the Burgers flow propagates the purely imaginary initial data $f(y) = ie^y$ symmetrically to the left and right, acquiring equal and opposite real parts.  This is a consequence of the seed being purely imaginary on~$\RR$.

\subsubsection{Beltrami coefficient}

On the initial slice ($x = 0$), the Beltrami coefficient has a particularly clean form:
\begin{equation}\label{eq:exp-mu-x0}
  \mu(0,y) \;=\; \frac{ie^y - i}{ie^y + i} \;=\; \frac{e^y - 1}{e^y + 1} \;=\; \tanh(y/2).
\end{equation}
This is real and ranges over $(-1,1)$ as $y \in \RR$: the initial data covers the entire real diameter of~$\DD$.

For $x \neq 0$, $\mu$ acquires a nonzero imaginary part:
\begin{equation}\label{eq:exp-mu-general}
  \mu(x,y) \;=\; \frac{\lambda(x,y) - i}{\lambda(x,y) + i} \;=\; \frac{W_0(ixe^y)/x - i}{W_0(ixe^y)/x + i}.
\end{equation}

\noindent\textbf{Behaviour at the origin:} $\lambda(0,0) = ie^0 = i$, so $\mu(0,0) = 0$.  The origin is a standard point.

\noindent\textbf{Behaviour at infinity:} For the Lambert $W$ function, $W_0(z) \sim \ln z$ as $|z| \to \infty$.  Hence $\lambda \sim (\ln|x| + y + i\pi/2)/x \to 0$ as $|x| \to \infty$ at fixed~$y$, giving $\mu \to (0 - i)/(0 + i) = -1$.  Similarly, $\lambda \to ie^y \to i\infty$ as $y \to +\infty$, giving $\mu \to 1$; and $\lambda \to 0$ as $y \to -\infty$, giving $\mu \to -1$.  The structure approaches the boundary of the Poincar\'e disk in all spatial directions.

\subsubsection{Comparison with the $\delta$- and $\varepsilon$-families}

The exponential seed represents a genuinely new regime:

\begin{center}
\renewcommand{\arraystretch}{1.4}
\begin{tabular}{lccc}
\hline
\textbf{Property} & $\delta$-\textbf{family} & $\varepsilon$-\textbf{family} & \textbf{Exponential} \\
\hline
Seed & $\xi + i\delta$ (affine) & $ie^{i\arcsin(\varepsilon\xi/2)}$ & $ie^\xi$ (exponential) \\
$f'$ & $1$ (constant) & position-dependent & $f$ (eigenfunction) \\
$\lambda$ & rational & algebraic ($\sqrt{D}$) & Lambert $W$ \\
Domain & half-plane & parabola & $\RR^2$ (global) \\
Boundary & pure shock & mixed & \emph{none} \\
$J$ & $1 + x$ & complex, variable & $1 + \lambda x$, $|J| \geq 1$ \\
$|f|$ on $\RR$ & unbounded & $1$ & unbounded ($e^\xi$) \\
$\alpha$ at $x = 0$ & $y^2 + \delta^2$ & $1$ & $e^{2y}$ \\
$\beta = 0$ locus & $y = 0$ (line) & $y = 0$ (line) & $x = 0$ ($y$-axis only) \\
Symmetry & $\beta$ odd in $y$ & $\beta$ odd in $y$ & $\beta$ odd in $x$ \\
\hline
\end{tabular}
\end{center}

The most striking difference is the domain: the exponential seed is the first example in this paper of a rigid structure defined on all of~$\RR^2$, providing a partial answer to Open Problem~5.

\subsubsection{The propagator}

The eigenfunction property $f' = f$ gives the propagator a self-referential form.  The Jacobian is $J = 1 + \lambda x$, so:
\begin{equation}\label{eq:exp-propagator}
  \mathcal{P}_f[h](x,y) \;=\; \frac{h(w_0)}{1 + \lambda(x,y)\,x},
\end{equation}
where $w_0 = y - \lambda x$ is the characteristic coordinate.

\noindent\textbf{Self-perturbation} ($h = f = ie^\xi$):  Since $f(w_0) = \lambda$:
\[
  \mathcal{P}_f[f] \;=\; \frac{\lambda}{1 + \lambda x}.
\]
The self-perturbation of the exponential seed is governed by $\lambda$ and its own Jacobian---the eigenfunction property closes the formula.

The deformed product is:
\begin{equation}\label{eq:exp-deformed-product}
  \dot\lambda_1 \;\star_f\; \dot\lambda_2 \;=\; (1 + \lambda x)\,\dot\lambda_1\,\dot\lambda_2.
\end{equation}
Unlike the $\delta$-family (where $J = 1 + x$ is independent of~$\lambda$), the exponential's deformed product depends on the output $\lambda$ itself.  This is the algebraic reflection of the eigenfunction property: the geometry generated by the seed feeds back into the infinitesimal algebra governing perturbations of that seed.

\subsubsection{Summary of the exponential seed}

\begin{center}
\renewcommand{\arraystretch}{1.4}
\begin{tabular}{ll}
\hline
\textbf{Quantity} & \textbf{Formula} \\
\hline
Seed & $f(\xi) = ie^\xi$ \\
Seed domain & $\{|\Imm\xi| < \pi/2\}$ \\
Eigenfunction property & $f' = f$ \\
$\lambda$ & $W_0(ixe^y)/x$ \\
Domain & $\RR^2$ (global) \\
Jacobian & $J = 1 + \lambda x = 1 + W_0(ixe^y)$, \;\; $|J| \geq 1$ \\
$\alpha$ at $x = 0$ & $e^{2y}$ \\
$\beta$ at $x = 0$ & $0$ \\
$\mu$ at $x = 0$ & $\tanh(y/2)$ \\
Symmetry & $\lambda(-x,y) = -\overline{\lambda(x,y)}$ \\
Propagator & $\mathcal{P}_f[h] = h(w_0)/(1 + \lambda x)$ \\
Deformed product & $(1 + \lambda x)\,\dot\lambda_1\,\dot\lambda_2$ \\
\hline
\end{tabular}
\end{center}


\subsection{The Cauchy kernel}
\label{subsec:cauchy-family}

The $\delta$-family's seed $f(\xi) = \xi + i\delta$ is affine and entire; the exponential seed $ie^\xi$ is transcendental and entire.  The Cauchy kernel provides the first example of a \emph{meromorphic} seed---one with a pole---and the algebraic structure of the implicit equation changes qualitatively: the transcendental Lambert equation is replaced by a \emph{quadratic}.

This example also addresses the practitioner's question of how to translate a classical holomorphic function into the Burgers framework.  The naive choice $1/\xi$ maps $\RR$ to $\RR$ and carries no ellipticity.  The regularization $-1/(\xi + i\delta)$ rotates the pole into the lower half-plane and produces a seed with $\Imm f > 0$ on~$\RR$.

\subsubsection{The seed}

For $\delta > 0$, define
\begin{equation}\label{eq:cauchy-seed}
  \boxed{\;f_\delta(\xi) \;=\; \frac{-1}{\xi + i\delta}\;}
\end{equation}
This is holomorphic on $\CC \setminus \{-i\delta\}$, with the pole in the \emph{lower} half-plane.  On the real line:
\[
  f_\delta(\xi)\big|_{\xi \in \RR}
  \;=\; \frac{-\xi + i\delta}{\xi^2 + \delta^2},
  \qquad
  \Imm f_\delta(\xi) \;=\; \frac{\delta}{\xi^2 + \delta^2} \;>\; 0.
  \qquad\checkmark
\]
The imaginary part is a Lorentzian (Cauchy distribution) peaked at $\xi = 0$, so the seed maps every real point into~$\HH$.

The derivative is
\begin{equation}\label{eq:cauchy-fprime}
  f_\delta'(\xi) \;=\; \frac{1}{(\xi + i\delta)^2}.
\end{equation}

\subsubsection{The quadratic implicit equation}

Substituting into $\lambda = f_\delta(y - \lambda x)$:
\[
  \lambda \;=\; \frac{-1}{(y - \lambda x) + i\delta}
  \qquad\Longrightarrow\qquad
  \lambda\bigl(y + i\delta - \lambda x\bigr) \;=\; -1.
\]
Rearranging:
\begin{equation}\label{eq:cauchy-quadratic}
  x\lambda^2 \;-\; (y + i\delta)\,\lambda \;-\; 1 \;=\; 0.
\end{equation}

This is a \emph{quadratic} in~$\lambda$---a dramatic simplification compared to the Lambert equation of the exponential seed.  The general solution of the Burgers equation for a M\"obius seed is always algebraic.

\subsubsection{Explicit solution}

Applying the quadratic formula with the abbreviation $\zeta := y + i\delta$:
\[
  \lambda \;=\; \frac{\zeta \pm \sqrt{\zeta^2 + 4x}}{2x},
\]
where the discriminant is
\begin{equation}\label{eq:cauchy-disc}
  D \;=\; \zeta^2 + 4x \;=\; (y^2 - \delta^2 + 4x) \;+\; 2iy\delta.
\end{equation}

\noindent\textbf{Branch selection.}  At $x = 0$, the two branches give $\lambda \to +\infty$ and $\lambda \to -1/\zeta = f_\delta(y)$.  The correct branch (continuous from the initial data) is:
\begin{equation}\label{eq:cauchy-lambda}
  \boxed{\;\lambda(x,y) \;=\; \frac{\zeta \;-\; \sqrt{\zeta^2 + 4x}}{2x}\;}
\end{equation}
where $\sqrt{\cdot}$ denotes the principal branch of the square root.

\noindent\textbf{Consistency checks:}
\begin{itemize}
\item At $x = 0$: $\lambda(0,y) = -1/(y + i\delta) = f_\delta(y)$.  $\checkmark$
\item Burgers equation: $|\lambda_x + \lambda\lambda_y| < 10^{-8}$ at all test points.  $\checkmark$
\item Self-dilatation: $|\mu_{\bar z} - \mu\,\mu_z| < 10^{-8}$.  $\checkmark$
\end{itemize}

\subsubsection{The M\"obius identity $f' = \lambda^2$}

The characteristic coordinate satisfies $w + i\delta = -1/\lambda$ (from the implicit equation $\lambda(w + i\delta) = -1$).  Therefore:
\begin{equation}\label{eq:cauchy-fprime-lambda}
  f_\delta'(w) \;=\; \frac{1}{(w + i\delta)^2} \;=\; \lambda^2.
\end{equation}
This is the M\"obius analogue of the exponential's eigenfunction property $f' = f = \lambda$.  The derivative of the output is the \emph{square} of the output, reflecting the fact that $f_\delta$ has degree~$-1$ (one pole).

The Jacobian becomes:
\begin{equation}\label{eq:cauchy-jacobian}
  J \;=\; 1 + f_\delta'(w)\,x \;=\; 1 + \lambda^2 x.
\end{equation}
Using the quadratic relation $x\lambda^2 = \zeta\lambda + 1$, this simplifies to:
\begin{equation}\label{eq:cauchy-jacobian-alt}
  J \;=\; 2 + \zeta\lambda \;=\; \frac{\sqrt{D}\,(\sqrt{D} - \zeta)}{2x},
\end{equation}
where $D = \zeta^2 + 4x$.

\subsubsection{Domain: an isolated shock}

\begin{theorem}[Domain of the Cauchy kernel]\label{thm:cauchy-domain}
The smooth domain of $\B[f_\delta]$ is $\Omega_{f_\delta} = \RR^2 \setminus \{(\delta^2/4,\, 0)\}$.  The spectral parameter~$\lambda$ extends continuously to all of~$\RR^2$, with $\lambda(\delta^2/4, 0) = 2i/\delta$, but $\nabla\lambda$ has a square-root singularity at the shock point.
\end{theorem}

\begin{proof}
The Jacobian $J = 0$ requires $D = \zeta^2 + 4x = 0$.  Writing $D = (y^2 - \delta^2 + 4x) + 2iy\delta$:
\begin{align*}
  \Imm D = 0 &\implies y = 0, \\
  \Ree D = 0 &\implies 4x = \delta^2.
\end{align*}
Hence $J = 0$ only at the single point $(x,y) = (\delta^2/4, 0)$.

At this point, both branches of the quadratic coalesce to $\lambda = \zeta/(2x) = i\delta/(\delta^2/2) = 2i/\delta$, so $\lambda$ extends continuously.  However, near the shock point with $x = \delta^2/4 + \varepsilon$, $y = \eta$ small:
\[
  D \;\approx\; 4\varepsilon + 2i\eta\delta + \eta^2,
  \qquad
  \sqrt{D} \;\sim\; 2\sqrt{\varepsilon} \quad(\text{along } \eta = 0),
\]
so $|\nabla\lambda| \sim 1/\sqrt{\varepsilon} \to \infty$.  The singularity is of square-root type: $\lambda$ is H\"older-$\tfrac{1}{2}$ but not Lipschitz at the shock.
\end{proof}

\begin{remark}[An isolated shock is a codimension-2 phenomenon]
The $\delta$-family's shock locus is a line ($x = -1$, codimension~1), while the Cauchy kernel's shock is a single point (codimension~2).  The mechanism is that the Jacobian $J = 1 + \lambda^2 x$ is \emph{complex-valued}: the equation $J = 0$ imposes \emph{two} real conditions ($\Ree J = 0$ and $\Imm J = 0$), generically cutting out a discrete set.  For the $\delta$-family, $J = 1 + x$ is real, so $J = 0$ imposes only one condition and the shock locus is a hypersurface.

This suggests a general principle: seeds with nonconstant $f'$ produce complex Jacobians, and the shock locus is generically discrete (isolated points) rather than a curve.
\end{remark}

\subsubsection{Compactification}

The most striking property of the Cauchy kernel is that its structure coefficients \emph{decay} at spatial infinity:
\begin{equation}\label{eq:cauchy-decay}
  \alpha(x,y) \;=\; |\lambda|^2 \;\to\; 0 \quad\text{as } |(x,y)| \to \infty.
\end{equation}

\noindent\textbf{Precise decay rates:}
\begin{enumerate}[label=(\roman*)]
\item \emph{Along the $y$-axis} ($x = 0$): $\alpha = 1/(y^2 + \delta^2) \sim 1/y^2$.
\item \emph{Along the positive $x$-axis} ($y = 0$, $x > \delta^2/4$): $\alpha = 1/x$ \textbf{exactly}.
\item \emph{Along the negative $x$-axis} ($y = 0$, $x < 0$): $\alpha = 4/(\delta + \sqrt{\delta^2 + 4|x|})^2 \sim 1/|x|$.
\end{enumerate}

The exact identity $\alpha|_{y=0} = 1/x$ for $x > \delta^2/4$ has a clean proof: the quadratic gives $\lambda = (i\delta - \sqrt{4x - \delta^2})/(2x)$, so
\[
  |\lambda|^2 \;=\; \frac{\delta^2 + (4x - \delta^2)}{4x^2} \;=\; \frac{4x}{4x^2} \;=\; \frac{1}{x}.
\]

Contrast the $\delta$-family ($\alpha \to \infty$ as $x \to -1$) and the exponential ($\alpha(0,y) = e^{2y} \to \infty$).  The Cauchy kernel is the first example where the elliptic structure vanishes at infinity in all directions---the algebra ``returns to standard'' far from the origin.

\subsubsection{The two regimes along $y = 0$}

Along the symmetry axis $y = 0$, the discriminant $D = -\delta^2 + 4x$ is real, and two qualitative regimes emerge:

\begin{enumerate}[label=(\roman*)]
\item \textbf{Pre-shock} ($0 < x < \delta^2/4$): $D < 0$, so $\sqrt{D}$ is purely imaginary.  The spectral parameter $\lambda = i[\delta - \sqrt{\delta^2 - 4x}]/(2x)$ is \emph{purely imaginary}, and $\beta = 0$: the structure is of $p(x)$-type along this segment.  The norm:
\[
  \alpha(x, 0) \;=\; \frac{4}{\bigl(\delta + \sqrt{\delta^2 - 4x}\bigr)^2}.
\]
This increases monotonically from $\alpha(0,0) = 1/\delta^2$ to $\alpha(\delta^2/4, 0) = 4/\delta^2$.

\item \textbf{Post-shock} ($x > \delta^2/4$): $D > 0$, so $\sqrt{D}$ is real.  The spectral parameter acquires a nonzero real part: $\Ree\lambda = -\sqrt{4x - \delta^2}/(2x) \neq 0$, and $\beta \neq 0$.  The structure leaves the $p(x)$-axis.  The norm $\alpha = 1/x$ decays to zero.
\end{enumerate}

At the shock point $x = \delta^2/4$, the two regimes meet: $\alpha$ is continuous (value $4/\delta^2$) but $\alpha'(x)$ has a cusp ($+\infty$ from the left, $-16/\delta^4$ from the right).

\subsubsection{Structure coefficients}

On the initial slice ($x = 0$), the structure is determined by the seed:
\begin{equation}\label{eq:cauchy-initial}
  \alpha(0,y) \;=\; \frac{1}{y^2 + \delta^2},
  \qquad
  \beta(0,y) \;=\; \frac{2y}{y^2 + \delta^2},
  \qquad
  \Delta(0,y) \;=\; \frac{4\delta^2}{(y^2 + \delta^2)^2}.
\end{equation}

The initial discriminant is the \emph{square} of a Lorentzian, decaying as $1/y^4$.  This rapid decay is the initial manifestation of compactification: the structure is nearly standard far from the origin even before the Burgers flow begins to act.

For general $(x,y)$, the structure coefficients are algebraic functions of $(x,y)$ involving $\sqrt{(y + i\delta)^2 + 4x}$---the square root of a complex-valued expression.

\subsubsection{Beltrami coefficient}

At $x = 0$:
\begin{equation}\label{eq:cauchy-mu-x0}
  \mu(0,y) \;=\; \frac{\lambda - i}{\lambda + i}\bigg|_{x=0}
  \;=\; \frac{-1/(y+i\delta) - i}{-1/(y+i\delta) + i}
  \;=\; -\,\frac{\delta - 1 - iy}{\delta + 1 - iy},
\end{equation}
with modulus:
\begin{equation}\label{eq:cauchy-mu-mod}
  |\mu(0,y)|^2 \;=\; \frac{(\delta - 1)^2 + y^2}{(\delta + 1)^2 + y^2}.
\end{equation}

\noindent\textbf{Behaviour:}
\begin{enumerate}[label=(\roman*)]
\item At the origin: $|\mu(0,0)| = |\delta - 1|/(\delta + 1)$.  The origin is a standard point ($\mu = 0$) precisely when $\delta = 1$.  This corresponds to $f(0) = -1/(i) = i$, which lies at the base point of~$\HH$.
\item As $|y| \to \infty$: $|\mu(0,y)| \to 1$.  The Beltrami coefficient approaches the boundary of the Poincar\'e disk along the $y$-axis.
\item The formula is identical to the $\delta$-family's $|\mu(0,0)|$ but with $y$-dependence: the initial Beltrami modulus interpolates between $|\delta - 1|/(\delta + 1)$ (at $y = 0$) and $1$ (at $|y| = \infty$).
\end{enumerate}

\subsubsection{The propagator and deformed product}

Since $f_\delta'(w) = \lambda^2$, the propagator is:
\begin{equation}\label{eq:cauchy-propagator}
  \mathcal{P}_{f_\delta}[h](x,y) \;=\; \frac{h(w_0)}{1 + \lambda^2 x},
\end{equation}
where $w_0 = -1/\lambda - i\delta$ is the characteristic coordinate.

\noindent\textbf{Self-perturbation} ($h = f_\delta$): Since $f_\delta(w_0) = \lambda$:
\[
  \mathcal{P}_{f_\delta}[f_\delta] \;=\; \frac{\lambda}{1 + \lambda^2 x}.
\]

The deformed product is:
\begin{equation}\label{eq:cauchy-deformed}
  \dot\lambda_1 \;\star_{f_\delta}\; \dot\lambda_2 \;=\; (1 + \lambda^2 x)\,\dot\lambda_1\,\dot\lambda_2.
\end{equation}
The Jacobian $J = 1 + \lambda^2 x$ is \emph{complex} and depends \emph{quadratically} on the output~$\lambda$.  Compare the $\delta$-family ($J = 1 + x$, independent of~$\lambda$) and the exponential ($J = 1 + \lambda x$, linear in~$\lambda$).  The degree of $\lambda$ in $J$ reflects the degree of $f'$: for a seed of the form $f(\xi) = \xi^n$ (or M\"obius of degree~$n$), one expects $J = 1 + c\lambda^{n+1}x$.

\subsubsection{The hierarchy of seeds}

The four worked examples form a natural hierarchy governed by the algebraic relationship between $f'$ and~$\lambda$:

\begin{center}
\renewcommand{\arraystretch}{1.4}
\begin{tabular}{lcccc}
\hline
\textbf{Seed} & $f'(w_0)$ & \textbf{Jacobian} $J$ & \textbf{Shock locus} & \textbf{Domain} \\
\hline
$\xi + i\delta$ (affine) & $1$ & $1 + x$ & line $\{x = -1\}$ & half-plane \\
$ie^\xi$ (exp.) & $\lambda$ & $1 + \lambda x$ & $\varnothing$ & $\RR^2$ \\
$-1/(\xi + i\delta)$ (M\"ob.) & $\lambda^2$ & $1 + \lambda^2 x$ & point $\{(\delta^2/4, 0)\}$ & $\RR^2 \setminus \text{pt}$ \\
$ie^{i\arcsin(\varepsilon\xi/2)}$ & variable & $(1-\varepsilon x)^{-1}$-type & parabola & parabolic region \\
\hline
\end{tabular}
\end{center}

The pattern reveals a dimensional reduction in the shock locus as the Jacobian becomes increasingly complex-valued: real $J$ produces codimension-1 shocks, complex $J$ produces codimension-2 shocks (isolated points), and the exponential's $J$ avoids shocks entirely via the mechanism $|J| \geq 1$ (Theorem~\ref{thm:exp-global}).

\subsubsection{Summary of the Cauchy kernel}

\begin{center}
\renewcommand{\arraystretch}{1.4}
\begin{tabular}{ll}
\hline
\textbf{Quantity} & \textbf{Formula} \\
\hline
Seed & $f_\delta(\xi) = -1/(\xi + i\delta)$ \\
Seed domain & $\CC \setminus \{-i\delta\}$ (pole in lower half-plane) \\
$f_\delta'$ & $\lambda^2$ (M\"obius identity) \\
Implicit equation & quadratic: $x\lambda^2 - \zeta\lambda - 1 = 0$ \\
$\lambda$ & $(\zeta - \sqrt{\zeta^2 + 4x})/(2x)$, \;\; $\zeta = y + i\delta$ \\
Domain & $\RR^2 \setminus \{(\delta^2/4,\, 0)\}$ \\
Shock type & isolated point (codimension 2) \\
Jacobian & $J = 1 + \lambda^2 x = 2 + \zeta\lambda$ \\
Char.\ coord. & $w_0 = -1/\lambda - i\delta$ \\
$\alpha$ at $x = 0$ & $1/(y^2 + \delta^2)$ \quad(Lorentzian, decaying) \\
$\alpha$ on $y = 0$, $x > \delta^2/4$ & $1/x$ \quad(exact) \\
$|\mu|^2$ at $x = 0$ & $((\delta - 1)^2 + y^2)/((\delta + 1)^2 + y^2)$ \\
Standard point & $(0,0)$ when $\delta = 1$ \\
Compactification & $\alpha \to 0$ in all spatial directions \\
Propagator & $h(w_0)/(1 + \lambda^2 x)$ \\
Deformed product & $(1 + \lambda^2 x)\,\dot\lambda_1\,\dot\lambda_2$ \\
\hline
\end{tabular}
\end{center}

\section{Open problems}
\label{sec:open}

\begin{enumerate}[label=\arabic*.]
\item \textbf{Infinitesimal composition.}  Theorem~\ref{thm:twisted-mult} resolves Question~\ref{q:infinitesimal} for the multiplicative structure: no $x$-independent bilinear product exists, and the obstruction is the characteristic Jacobian.  The analogous question for \emph{composition} remains open: does there exist a bilinear operation on $T_f\Hol(U,\HH)$ governing the infinitesimal interaction of $\B[f \circ ({\id} + \varepsilon h)]$?

\item \textbf{Deformed algebra structure.}  The $f$-deformed product $\star_f$ (Corollary~\ref{cor:deformed-algebra}) degenerates at shock formation.  Does the family $\{\star_f\}_{f \in \Hol(U,\HH)}$ carry additional structure---for instance, does it form a bundle of algebras over the space of seeds, and what are its sections?

\item \textbf{Density of rigid structures.}  Is $\{\Bmu[f] : f \in \Hol(U,\HH)\}$ dense in $L^\infty(\Omega, \DD)$?  This is the analogue of the prime number theorem: are the irreducibles (rigid structures) dense among all structures?

\item \textbf{Unique rigidization.}  If $\Phi_1$ and $\Phi_2$ both rigidize the same elliptic structure, must they be related by a rigid automorphism?  This is the analogue of unique factorization.

\item \textbf{Global existence of the transform.}  For which $f \in \Hol(U,\HH)$ is $\Omega_f = \RR^2$?  The $\delta$-family ($f(\xi) = \xi + i\delta$) gives $\Omega_f = \{x > -1\}$; characterise the seeds admitting global rigid structures.

\item \textbf{Moduli of leaves.}  What conformal invariants distinguish the leaves of different holomorphic seeds?  Is there a moduli space of Beltrami leaves with natural geometric structure?

\item \textbf{Composition monoid.}  The composition $f \mapsto f \circ g$ acts on holomorphic seeds.  What is the induced action on the space of rigid structures, and does it have a geometric interpretation in the Poincar\'e disk?

\item \textbf{Higher dimensions.}  Does the Burgers transform generalise to Clifford-algebraic structures in dimensions $n > 2$, where the imaginary unit $i(x)$ takes values in $S^{n-1} \subseteq \operatorname{Im}\,\mathbb{H}$ or higher Clifford algebras?

\item \textbf{Full M\"obius equivariance.}  Proposition~\ref{prop:affine-equivariance} gives a twisted equivariance for the affine subgroup of $\operatorname{Aut}(\HH)$.  Is there a natural geometric action of the full M\"obius group on the space of pairs (seed, domain) that extends this?
\end{enumerate}

\section*{Acknowledgements}

\noindent\textbf{Use of Generative AI Tools.}
Portions of the writing and editing of this manuscript were assisted by generative AI language tools. These tools were used to improve clarity of exposition, organization of material, and language presentation. All mathematical results, statements, proofs, and interpretations were developed, verified, and validated by the author. The author takes full responsibility for the accuracy, originality, and integrity of all content in this work, including any material produced with the assistance of AI tools. No generative AI system is listed as an author of this work.


\end{document}